\documentclass{amsart}
\usepackage{amsmath}
\usepackage{amssymb}
\usepackage{enumitem}
\usepackage{hyperref}
\usepackage{subcaption}
\usepackage{tikz}

\newtheorem{theorem}{Theorem}[section]
\newtheorem{corollary}[theorem]{Corollary}
\newtheorem{definition}[theorem]{Definition}
\newtheorem{lemma}[theorem]{Lemma}
\newtheorem{remark}[theorem]{Remark}
\numberwithin{equation}{section}

\SetEnumitemKey{runin}{align=left, font=\mdseries\rmfamily\upshape, label=(\arabic*), itemindent=0.5em, labelindent=0em, labelsep=0.5em, labelwidth=0em, leftmargin=0em, listparindent=\parindent}

\begin{document}
\title{\textbf{A finite theorem for Ahlfors' covering surface theory}}
\author{Tian~Run Lin \& Yun~Ling Chen \& Guang~Yuan Zhang}
\address{Department of Mathematical Sciences, Tsinghua University, Beijing
100084, P. R. China.}
\address{\slshape Email: ltr17@mails.tsinghua.edu.cn}
\address{\slshape Email: chenyl20@mails.tsinghua.edu.cn}
\address{\slshape Email: gyzhang@mail.tsinghua.edu.cn}
\thanks{Project 10971112 and 12171264 supported by NSFC}

\begin{abstract}
Ahlfors' theory of covering surfaces is one of the major mathematical
achievement of last century. The most important part of his theory is the
Second Fundamental Theorem (SFT). We are interested in the relation of
errors of Ahlfors' SFT with the same boundary curve.

In this paper we will prove a result which is used to establish the best
bound of the constant in Ahlfors' SFT (in \cite{Zh}).

Precisely speaking, we will prove that for any surface $\Sigma\in\mathcal{F}%
_r(L,m)$, a new surface $\Sigma_1$ can be constructed based on it, such that
$R(\Sigma_1)\ge R(\Sigma)$ and $L(\partial\Sigma_1)\le L(\partial\Sigma)$,
where $R(\Sigma)$ is Ahlfors' error term and $L(\partial\Sigma)$ is the
boundary length of the surface $\Sigma$, and the covering degree of $%
\Sigma_1 $ has an upper bound independent of surfaces. Meanwhile, this
conclusion suggests that the supremum of $H(\Sigma)=R(\Sigma)/L(\partial%
\Sigma)$ can be achieved by surfaces in the space $\mathcal{F}_r^{\prime
}(L,m)$.
\end{abstract}

\maketitle
\tableofcontents

\section{Introduction}

\label{sec:intro}

In the study of the value distribution theory, Nevanlinna's theory plays a
major role as we know (\cite{MR0164038,MR3751331,MR1555200,MR1301781}). And
Ahlfors' theory of covering surfaces, which is parallel to the theory of
Nevanlinna, can also give us useful methods to discuss meromorphic functions
(\cite{MR1555403,MR994468,MR0012669,MR1786560,MR0164038}). Ahlfors' theory
builds upon geometrical views and gives the Second Fundamental Theorem
parallel to, but different from Nevanlinna's theory, whose form is similar
to a type of isoperimetric inequality. For this reason Ahlfors' theory leads
to the interesting geometrical constant called Ahlfors constant, which have
been discussed in \cite{MR4125732,MR3004780}.

We need recall some basic definitions and notations. The extended complex
plane $\overline{\mathbb{C}}=\mathbb{C}\cup\{\infty\}$ can be naturally
identified with the unit sphere $S$ in $\mathbb{R}^3$ via the stereographic
projection as introduced in \cite{MR510197}. Following the Euclidean metric
on $\mathbb{R}^3$, the sphere $S$ has the standard metric and the
corresponding K\"{a}hler form. We can represent them with coordinate in $%
\mathbb{C}$:
\begin{align*}
ds^2&=\frac{4}{(1+\lvert z\rvert^2)^2}\lvert dz\rvert^2, & \omega&=\frac{2%
\sqrt{-1}}{\pi}\frac{dz\wedge d\bar{z}}{(1+\lvert z\rvert^2)^2}, & &z\in%
\mathbb{C}.
\end{align*}
Then the spherical length and the spherical area can be defined. We will use
$L$ and $A$ to denote them respectively.

Let $U$ be $\overline{\mathbb{C}}$, or be a domain in $\overline{\mathbb{C}}$
enclosed by a finite number of Jordan curves, and let $f$ be a meromorphic
function defined on $\overline{U}$. Then $f$ can be regard as a holomorphic
mapping from $\overline{U}\subset \overline{\mathbb{C}}$ into $S$, via
stereographic projection. The pair $\Sigma =(f,\overline{U})$ can be
regarded as a holomorphic covering surface spread over the sphere $S$, and
the boundary $\partial \Sigma $ of $\Sigma $ is defined to be the pair $%
\partial \Sigma =(f,\partial U),$ where $\partial U$ is the boundary of $U$.
The length $L(\partial \Sigma )=L(f,\partial U)$ and area $A(\Sigma )=A(f,U)$
can be given by
\begin{align*}
L(f,\partial U)& =\int_{\partial U}f^{\ast }ds=\int_{\partial U}\frac{%
2\lvert f^{\prime }(z)\rvert }{1+\lvert f(z)\rvert ^{2}}\lvert dz\rvert , \\
A(f,U)& =\int_{U}f^{\ast }\omega =\frac{2\sqrt{-1}}{\pi }\int_{U}\frac{%
\lvert f^{\prime }(z)|^{2}}{(1+\lvert f(z)\rvert ^{2})^{2}}dz\wedge d\bar{z},
\end{align*}%
where $f^{\ast }$ means the pull-back by $f$. Note that the length $%
L(f,\partial U)$ may be infinite.

Compared to holomorphic covering surfaces, we define covering surfaces as in
Ahlfors' paper \cite{MR1555403}: a covering surface over $S$ is built up by
a finite number of one sheeted closed topological triangular domains.
Equivalently to say, it can be defined as

\begin{definition}
\label{def:surface} A \emph{covering surface} over $S$ is a pair $(f,%
\overline{U})$, where $U$ is $\overline{\mathbb{C}}$ or a domain in $%
\overline{\mathbb{C}}$ enclosed by a finite number of Jordan curves and $f:%
\overline{U}\to S$ is an orientation-preserving, continuous, open and
finite-to-one mapping (OPCOFOM), which in fact means that $f$ can be
extended to be OPCOFO on a neighborhood of $\overline{U}$. Here
orientation-preserving means that when $S$ is identified with $\overline{%
\mathbb{C}}$ via stereograpchic projection, $f$ is orientation-preserving on
$\mathbb{C}\cap(\overline{U}\setminus f^{-1}(\infty))$.

We denote by $\mathbf{F}$ the set of all covering surfaces such that for
each $(f,\overline{U})\in\mathbf{F}$, $U$ is a Jordan domain. Then we denote
by $\mathbf{F}(L)$ the subspace of $\mathbf{F}$ such that for each $\Sigma\in%
\mathbf{F}(L)$, $L(\partial\Sigma)\le L$.
\end{definition}

By Stoilow's theorem, a covering surface can be regarded as a holomorphic
covering surface, up to an orientation-preserving homeomorphic (OPH)
transform of the domain of definition. That is,

\begin{theorem}[Sto\"{\i}low, see \protect\cite{MR0082545} pp.120--121]
\label{thr:stoilow} Let $U$ be a domain on $\overline{\mathbb{C}}$ and $%
f:U\to S$ be an OPCOFOM. Then there exists a domain $U_1$ in $\overline{%
\mathbb{C}}$ and an OPH $\varphi:U_1\to U$, such that $f_1=f\circ\varphi:U_1%
\to S$ is a non-constant holomorphic mapping on $U_1$.
\end{theorem}

By the compactness of $\overline{\mathbb{C}}$, if $U=\overline{\mathbb{C}}$
we must have $U_1=\overline{\mathbb{C}}$ as well.

Now we can state Ahlfors' famous theorem published in \cite{MR1555403} in
1935, parallel to Nevanlinna's Second Fundamental Theorem.

\begin{theorem}[Ahlfors' Second Fundamental Theorem]
Let $q\ge3$, and $E_q=\{a_1,a_2,\dots,a_q\}$ be a set consisting of distinct
$q$ points on $S$. Then there exists a constant $h>0$ that depends only on $%
E_q$, such that
\begin{equation*}
(q-2)A(\Sigma)\le 4\pi\sum_{j=1}^{q}\overline{n}(\Sigma,a_j)+hL(\partial%
\Sigma),
\end{equation*}
for any $\Sigma=(f,\overline{U})\in\mathbf{F}$, where $\overline{n}%
(\Sigma,a_j)$ is the cardinality of $f^{-1}(a_j)\cap U$ (ignoring
multiplicity).
\end{theorem}

We will assume $E_q=\{a_1,a_2,\dots,a_q\}$ be the set of arbitrarily given
distinct $q$ points, $q\ge3$, and for any surface $\Sigma=(f,\overline{U})$,
write
\begin{equation*}
\overline{n}(\Sigma,E_q)=\sum_{j=1}^q\overline{n}(\Sigma,a_j).
\end{equation*}
We call $E_q$ the \emph{special set} and the points in $E_q$ the \emph{%
special points}.

The inequality in Ahlfors' Second Fundamental Theorem, which makes a
relation between geometric quantity and the counting function $\overline{n}%
(\Sigma,a_j)$, works effectively in research of the value distribution
theory of meromorphic functions obviously. For example \cite{MR0164038}, if $%
q=3$ and $f$ is a meromorphic function on $\mathbb{C}$ that omits $a_1$, $%
a_2 $, $a_3$, then this inequality tells us $A(R)\le hL(R)$, where $%
A(R)=A(f,\Delta_{R})$, $L(R)=L(f,\partial\Delta_{R})$ and $\Delta_{R}=\{z\in%
\mathbb{C}:\lvert z\rvert<R\}$. On the other hand, by the expression of $L$
and $A$ we can easily deduce that $L^2(R)\le2\pi^2RA^{\prime }(R)$. Then $%
A^2(R)\le2\pi^2h^2RA^{\prime }(R)$, which is a type of Gronwall inequality
and finally gives $A(R)=0$. It proves Picard's theorem, which asserts that
any meromorphic function on $\mathbb{C}$ that omits three distinct values
must be constant.

As we see, the existence of constant $h$ in Ahlfors' Second Fundamental
Theorem has played a big role in the value distribution theory. However, the
minimal possible value of this constant has not been sufficiently studied up
to now. This minimal value is called \emph{Ahlfors constant} in \cite%
{MR4125732}.

This problem can be traced back to the early 1940s, when J. Dufresnoy first
gave a numerical estimate of $h$ in \cite{MR0012669} as follows.

\begin{theorem}[Dufresnoy \protect\cite{MR0012669}]
For any $\Sigma=(f,\overline{U})\in\mathbf{F}$,
\begin{equation*}
(q-2)A(\Sigma)\le4\pi\overline{n}(\Sigma,E_q)+(q-2)\frac{6\pi}{\delta_{E_q}}%
L(\partial\Sigma),
\end{equation*}
where $\delta_{E_q}=\min_{1\le i<j\le q}d(a_i,a_j)$.
\end{theorem}

In 2009, the precise bound for $h$ have been identified in a special case as
follows.

\begin{theorem}[Zhang \protect\cite{MR3004780}]
For any $\Sigma\in\mathbf{F}$ with $\overline{n}(\Sigma,\{0,1,\infty\})=0$,
\begin{equation*}
A(\Sigma)\le h_0L(\partial\Sigma),
\end{equation*}
where
\begin{equation*}
h_0=\max_{\theta\in\lbrack0,\pi/2]}\left(\frac{(\pi+\theta)\sqrt{%
1+\sin^2\theta}}{\arctan\frac{\sqrt{1+\sin^2\theta}}{\cos\theta}}%
-\sin\theta\right),
\end{equation*}
and there exists a sequence $\{\Sigma_{n}\}=\{(f_n,\overline{\Delta})\}$ in $%
\mathbf{F}$ with $\overline{n}(\Sigma_n,\{0,1,\infty\})=0$ such that $%
A(\Sigma_{n})/L(\partial\Sigma_{n})\to h_0$ as $n\to\infty$.
\end{theorem}

\begin{definition}
We will write
\begin{equation*}
R(\Sigma,E_q)=(q-2)A(\Sigma)-4\pi\overline{n}(\Sigma,E_q)
\end{equation*}
and
\begin{equation*}
H(\Sigma,E_q)=\frac{R(\Sigma,E_q)}{L(\partial\Sigma)}.
\end{equation*}

When $E_q$ is fixed, we can write $\overline{n}(\Sigma)=\overline{n}%
(\Sigma,E_q)$, $R(\Sigma)=R(\Sigma,E_q)$, and $H(\Sigma)=H(\Sigma,E_q)$ for
convenience.
\end{definition}

Then, computing the minimal value $h$ in Ahlfors' Second Fundamental Theorem
is equivalent to find the supremum
\begin{equation*}
H_0=\sup\{H(\Sigma):\Sigma\in\mathbf{F}\}.
\end{equation*}

It is proved in \cite{Zh1} that the supremum $H_0$ of $H(\Sigma)$ cannot be
realized by any surface of $\mathbf{F}$. But as pointed in \cite{MR3004780},
to study $H_0$, one has to study the supremum
\begin{equation*}
H_L=\sup\{R(\Sigma)/L(\partial\Sigma):\Sigma\in\mathbf{F}(L)\}
\end{equation*}
for given $L>0,$ and has to find an extremal surface $\Sigma_L$ in $\mathbf{F%
}(L)$ which assume $H_L$: $H(\Sigma_L)=H_L$.

The inequality in Ahlfors' Second Fundamental Theorem can be regarded as a
type of isoperimetric inequality. Let us recall that how could we prove the
most basic isoperimetric inequality on Euclidean plane. Any closed curve can
be approximated by polygons, and for fixed $m$, the space composed of all $m$%
-polygons possesses certain compactness so we can easily show that regular
polygons have the largest area when the perimeter and edge number are fixed.
Then all the work was finished.

To study the existence of $\Sigma_L$, it is natural to consider the ``$m$%
-polygons'' on the Riemann sphere $S$, whose boundary can be divided into $m$
simple circular arcs, as in \cite{Zh}. On the other hand, any Jordan domain $%
U$ is conformally equivalent to $\Delta=\{z\in\mathbb{C}:\lvert z\rvert<1\}$%
, the unit disk centered at the origin. This leads to a corresponding
between surfaces defined on $\overline{U}$ and $\overline{\Delta}$. Of
course it will be convenient if we only discuss surfaces defined on the
closed unit disk $\overline{\Delta}$. Now we give the definition of ``$m$%
-polygon'' as in \cite{Zh}:

\begin{definition}
\label{def:FLM} We denote by $\mathcal{F}(L,m)$ the subset of $\mathbf{F}(L)$%
, which consists of all surfaces $\Sigma=(f,\overline{\Delta})\in\mathbf{F}%
(L)$ satisfying the following conditions:

(1) $\partial \Delta $ and $\partial \Sigma $ have partitions
\begin{equation}
\partial \Delta =\alpha _{1}(\mathfrak{p}_{1},\mathfrak{p}_{2})+\alpha _{2}(%
\mathfrak{p}_{2},\mathfrak{p}_{3})+\dots +\alpha _{m}(\mathfrak{p}_{m},%
\mathfrak{p}_{1})  \label{def:FLM-1}
\end{equation}%
and
\begin{equation}
\partial \Sigma =c_{1}(p_{1},p_{2})+c_{2}(p_{2},p_{3})+\dots
+c_{m}(p_{m},p_{1}),  \label{def:FLM-2}
\end{equation}%
such that for each $\alpha _{j}$, $c_{j}=(f,\alpha _{j})$ is a convex simple
circular arc;

(2) for each $\alpha _{j}$, there exists a neighborhood $U_{j}$ of $\alpha
_{j}^{\circ }=\alpha _{j}\setminus \{\mathfrak{p}_{j-1},\mathfrak{p}_{j}\}$
(means the interior of $\alpha _{j}$) in $\overline{\Delta }$, such that $f$
restricted to $U_{j}$ is an OPH.
\end{definition}

We call (\ref{def:FLM-1}) and (\ref{def:FLM-2}) $\mathcal{F}(L,m)$%
-partitions of $\partial\Sigma$.

\begin{remark}
\label{rm:FLMconvex} In this definition, we say that the circular arcs $c_j$
is convex, which means that for any $x\in\alpha_j^{\circ}$, there exists a
neighborhood $U$ in $\overline{\Delta}$, and a closed hemisphere $\overline{%
S^{\prime }}$ on $S$, such that $f(x)\in\partial\overline{S^{\prime }}$, and
$f(U)\subset\overline{S^{\prime }}$. More detailed discussion of convex
curves will be mentioned in Definition \ref{def:oriconv}.

Note that $\mathcal{F}(L,m)$-partitions are not only defined for the
boundary $\partial\Sigma$, but also involve the information of $f$ near the
boundary $\partial\Delta$.
\end{remark}

The space $\mathcal{F}(L,m)$ consists of some relatively simple surfaces,
but actually this space is still too large to study. Many surfaces in $%
\mathcal{F}(L,m)$ have bad performance and it is difficult to talk about the
convergence. Our target in this paper is to reduce the space $\mathcal{F}%
(L,m)$ further, that means we need find the subspace such that $H(\Sigma)$
has the same supremum as $\mathcal{F}(L,m)$.

It is easy to see
\begin{equation*}
H_L=\sup\{H(\Sigma):\Sigma\in\mathbf{F}(L)\}=\sup\{H(\Sigma:\Sigma\in%
\mathcal{F}(L,m))\}
\end{equation*}
as in \cite{Zh}. Then to find extremal surfaces of $\mathbf{F}(L)$ is to
find extremal surfaces of $\mathcal{F}(L,m)$. Surfaces in $\mathcal{F}(L,m)$
may not be one sheeted over $S$ and the defining function may have branch
values. Then we define for each $\Sigma=(f,\overline{U})\in\mathbf{F}(L)$:
\begin{align*}
\deg_{\min}(\Sigma)&=\min\{\#f^{-1}(y):y\in S\setminus\partial\Sigma\text{
and }y\text{ is not a branch value}\}, \\
\deg_{\max}(\Sigma)&=\max\{\#f^{-1}(y):y\in S\setminus\partial\Sigma\text{
and }y\text{ is not a branch value}\}.
\end{align*}
For convenience, they will be also denoted by $\deg_{\min}(f)$ and $%
\deg_{\max}(f)$.

For a surface $(f,\overline{\Delta})\in\mathcal{F}(L,m)$, a point $%
p\in\Delta $ is a branch point if $f$ is not injective in any neighborhood
of $p$, and a point $p\in\partial\Delta$ is called a branch point of $f$ if $%
f$ is not injective on $\Delta\cap D$ for any neighborhood $D$ of $p$ in $%
\overline{\Delta}$. A point $q\in S$ is a branch value of $f$ if $q$ is the
image of some branch point.

In the study of the existence of extremal surfaces of $\mathcal{F}(L,m)$,
the subspaces $\mathcal{F}_r^{\prime }(L,m)$ and $\mathcal{F}_r(L,m)$ play
important roles, which is defined in \cite{Zh} as follows.

\begin{definition}
We define $\mathcal{F}_r^{\prime }(L,m)\subset\mathcal{F}(L,m)$, such that
for each $\Sigma=(f,\overline{\Delta})\in\mathcal{F}_r^{\prime }(L,m)$, the
following hold:

(1) $f$ has no branch value outside $E_{q}$;

(2) $\deg _{\max }(f)<d^{\ast }$, where $d^{\ast }=d^{\ast }(m,q)$ is a
constant which depends only on $m$ and $q$.

Moreover, we use $\mathcal{F}_{r}(L,m)$ to denote the subspace of $\mathcal{F%
}(L,m)$ consisting of all surfaces satisfying (1).
\end{definition}

In \cite{Zh}, the last author has proved the existence of extremal surfaces
of $\mathbf{F}(L)$. The starting point of the proof is the following two
theorems, which are only asserted there.

\begin{theorem}
\label{thr:branch} Let $\Sigma=(f,\overline{\Delta})\in\mathcal{F}(L,m)$ and
assume $H(\Sigma)\ge H_L-\frac{2\pi}{L(\partial\Sigma)}$. Then there exists
a surface $\Sigma_1=(f_1,\overline{\Delta})\in\mathcal{F}_r(L,m)$, such that
\begin{equation*}
L(\partial\Sigma_1)\le L(\partial\Sigma)
\end{equation*}
and
\begin{equation*}
H(\Sigma_1)\ge H(\Sigma).
\end{equation*}
\end{theorem}

\begin{theorem}
\label{thr:main} There exists a constant $d^{\ast}=d^{\ast}(q,m)$ depending
only on $m$ and $q$, such that for any $\Sigma\in\mathcal{F}_r(L,m)$, there
exists a surface $\Sigma_1=(f_1,\overline{\Delta})\in\mathcal{F}_r^{\prime
}(L,m)$ satisfying
\begin{equation*}
\partial\Sigma_1=\partial\Sigma
\end{equation*}
and
\begin{equation*}
H(\Sigma_1)=H(\Sigma).
\end{equation*}
\end{theorem}

The importance of Theorem \ref{thr:branch} and Theorem \ref{thr:main}, is
that they respectively imply

\begin{corollary}
\begin{equation*}
\sup_{\Sigma\in\mathcal{F}(L,m)}H(\Sigma)=\sup_{\Sigma\in\mathcal{F}%
_r(L,m)}H(\Sigma)=\sup_{\Sigma\in\mathcal{F}_r^{\prime }(L,m)}H(\Sigma).
\end{equation*}
\end{corollary}

Theorem \ref{thr:branch} has been proved in \cite{CYL}. The goal of this
paper is to prove Theorem \ref{thr:main}.

\section{Elementary properties of surfaces}

In this section, we will introduce some concepts about surfaces, and basic
methods, some of which have been mentioned in Section \ref{sec:intro}. For
rigorousness, let us review the geometry on the sphere firstly.

\begin{definition}
On the Riemann sphere $S$, a simple arc on a circle is called a \emph{simple
circular arc}, and a simple circular arc on a great circle is called a \emph{%
line segment}.
\end{definition}

The following result is obvious.

\begin{lemma}
For any two non-antipodal points on $S$, there exists exactly one great
circle passing through them. Any two distinct great circles on $S$ intersect
at exactly two points.
\end{lemma}

\begin{definition}
A \emph{curve} $\gamma$ on $S$ is a continuous mapping $\gamma:[0,1]\to S$.
We call $\gamma(0)$, $\gamma(1)$ the initial point and the terminal point
respectively.

The interior of $\gamma$ is denoted by $\gamma^{\circ}$ which is the
restriction of $\gamma$ to $(0,1)$. An arc of $\gamma$ is a curve given by $%
t\mapsto\gamma(s+lt)$ for some $0\le s<s+l\le1$.

$\gamma$ is called closed if $\gamma(0)=\gamma(1)$, and called simple if $%
\gamma(s)\ne\gamma(t)$ for any $0\le s<t\le1$ except $s=0$, $t=1$.
\end{definition}

\begin{remark}
\label{rm:curveeq} Sometimes we only consider the image of curves, and we
directly use set operators with them.

For two curves $\gamma_1$, $\gamma_2$, we say that $\gamma_1$ and $\gamma_2$
are equivalent if there is an increasing homeomorphism $\varphi:[0,1]\to[0,1]
$ such that $\gamma_1=\gamma_2\circ\varphi$. In this case, without causing
misunderstanding, we will not distinguish them and directly write $%
\gamma_1=\gamma_2$ for convenience.

Sometimes for emphasis, we use $\gamma(x_0,x_1)$ to represent a curve $%
\gamma $ with the initial point $x_0$ and the terminal point $x_1$.
\end{remark}

\begin{definition}
For a curve $\gamma$ on $S$, $-\gamma$ is defined by
\begin{align*}
(-\gamma)(t)&=\gamma(1-t), & &0\le t\le1,
\end{align*}
which represents the opposite curve of $\gamma$ from $\gamma(1)$ to $%
\gamma(0)$.

For two curves $\gamma_1$ and $\gamma_2$ with $\gamma_1(1)=\gamma_2(0)$,
their sum $\gamma_1+\gamma_2$ is defined by
\begin{align*}
(\gamma_1+\gamma_2)(t)&=\gamma_1(2t), & &0\le t\le1/2, \\
(\gamma_1+\gamma_2)(t)&=\gamma_2(2t-1), & &1/2\le t\le1,
\end{align*}
which represents the joined curve.
\end{definition}

According to Remark \ref{rm:curveeq}, the addition of curves satisfies the
associative law. However, $\gamma+(-\gamma)$ is nonvanishing, which shows a
point moving along the curve $\gamma$ and then returning back.

A class of curves which we often discuss consists of the boundaries of
Jordan domains. For this situation, it is necessary to define its
orientation and convexity.

\begin{definition}
\label{def:oriconv}(1) A domain $D$ on $S$ is called a \emph{Jordan domain},
when $\partial D$ is a \emph{Jordan curve}, say, $\partial D$ is a simple
and closed curve.

(2) For a Jordan domain $D$ on $S$, let $h$ be a M\"{o}bius transformation
with $h(D)\subset \Delta $. Then $\partial D$ is oriented by the
anticlockwise orientation of $\partial h(D)$. We also say $D$ is a domain
enclosed by $\partial D$ with this orientation.

(3) A domain $D$ on $S$ is called \emph{convex} if for any two non-antipodal
points $x_{1}$, $x_{2}$ in $D$, the shorter line segment connecting them is
contained in $D$. The closure of a convex Jordan domain is called a convex
closed domain. A Jordan curve on $S$ is called \emph{convex} if it encloses
a convex Jordan domain.

(4) Let $\gamma $ be a curve on $S$ and $t_{0}\in (0,1)$, $\gamma $ is
called (locally) \emph{convex} at $\gamma (t_{0})$, if $\gamma $ restricted
to a neighborhood of $t_{0}$ is an arc of some convex Jordan curve.
\end{definition}

For instance, the disk $\{z\in\overline{\mathbb{C}}:\lvert z\rvert>2\}$ is
viewed as a convex domain on $S$, and thus the circle $\lvert z\rvert=2$
oriented clockwise is also convex on $S$, conversely this circle oriented
anticlockwise is not convex. It is easy to see the definition here is
consistent with the previous one in Remark \ref{rm:FLMconvex}.

Now we introduce the concept of surface equivalence, where surfaces is
described in Definition \ref{def:surface}.

\begin{definition}
For two surfaces $\Sigma_1=(f_1,\overline{U_1})$ and $\Sigma_2=(f_2,%
\overline{U_2})$, we say $\Sigma_1$ and $\Sigma_2$ are equivalent, and write
\begin{equation*}
(f_1,\overline{U_1})\sim(f_2,\overline{U_2}),
\end{equation*}
when there is an orientation-preserving homeomorphism $\varphi:\overline{U_1}%
\to\overline{U_2}$ such that $f_1=f_2\circ\varphi$.

Moreover, for two pairs $(f_1,\gamma_1)$ and $(f_2,\gamma_2)$, where $%
\gamma_1$ and $\gamma_2$ are curves in $\overline{U_1}$ and $\overline{U_2}$
respectively, we write
\begin{equation*}
(f_1,\gamma_1)\sim(f_2,\gamma_2),
\end{equation*}
when $f_1\circ\gamma_1$ and $f_2\circ\gamma_2$ are equivalent curves.
\end{definition}

Obviously, if two surfaces $(f_1,\overline{U_1})$ and $(f_2,\overline{U_2})$
satisfy $(f_1,\overline{U_1})\sim(f_2,\overline{U_2})$, then $(f_1,\partial
U_1)\sim(f_2,\partial U_2)$. For equivalent surfaces $\Sigma_1$ and $%
\Sigma_2 $, we have $A(\Sigma_1)=A(\Sigma_2)$, $\overline{n}(\Sigma_1)=%
\overline{n}(\Sigma_2)$, and $L(\partial\Sigma_1)=L(\partial\Sigma_2)$.

By the theorem of Sto\"ilow (Theorem \ref{thr:stoilow}), for each surface $%
\Sigma=(f,\overline{U})$, there is a homeomorphism $\varphi:V_1\to V\supset%
\overline{U}$, where $V$ and $V_1$ are domains on $\overline{\mathbb{C}}$,
such that $f\circ\varphi$ is a meromorphic function on $V_1$. So any surface
is given by a meromorphic function in the sense of equivalence. $A(\Sigma)$,
$\overline{n}(\Sigma)$ and $L(\partial\Sigma)$ can be determined in this
way. With such equivalence, the definition of branch points can also be
determined.

The following conclusion is contained in Sto\"ilow's theorem.

\begin{lemma}
\label{lm:stoilowloc} Let $(f,\overline{U})$ be a surface and $x\in\overline{%
U}$. Then there exist homeomorphisms $\varphi:\Delta\to\varphi(\Delta)\subset%
\overline{\mathbb{C}}$ and $\psi:\Delta\to\psi(\Delta)\subset\overline{%
\mathbb{C}}$, such that $\varphi(0)=x$, $\psi(0)=f(x)$, $f$ can be extended
to $\varphi(\Delta)\cup\overline{U}$ with $f(\varphi(\Delta))=\psi(\Delta)$,
and $(\psi^{-1}\circ f\circ\varphi)(\zeta)=\zeta^{\omega}$ on $\Delta$,
where $\omega$ is a positive integer.
\end{lemma}

In other words, $f$ is locally a branched covering mapping. But its mapping
behavior near an interior point and a boundary point are quite different.

Let $(f,\overline{U})$ be a surface, where $U$ is $\overline{\mathbb{C}}$ or
a domain in $\overline{\mathbb{C}}$ enclosed by a finite number of Jordan
curves by Definition \ref{def:surface}. Then for any $x\in\overline{U}$, by
Lemma \ref{lm:stoilowloc} we can find the corresponding homeomorphism $%
\varphi$, $\psi$. If $x$ is an interior point say, $x\in U$, we can require $%
\varphi(\Delta)\subset U$ for extra, and for each $y\in\psi(\Delta)\setminus%
\{f(x)\}$, the pre-image $f^{-1}(y)$ has exactly $\omega$ points in $%
\varphi(\Delta)$.

If $x$ is a boundary point, say, $x\in \partial U$, we can require that $%
\varphi $ maps the diameter interval $(-1,1)\subset \Delta $ onto $\varphi
((-1,1))=\varphi (\Delta )\cap \partial U$ for extra, and maps the upper
half disk $\Delta ^{+}$ onto $\varphi (\Delta ^{+})=\varphi (\Delta )\cap U$%
. For $y\in \psi (\Delta )\setminus \{f(x)\}$, $f^{-1}(y)$ still has $\omega
$ points in $\varphi (\Delta )$, however, some of those points may be
located outside $\overline{U}$. Thus there are two possibilities need to be
discussed.

When $\omega$ is odd, $\zeta\mapsto\zeta^{\omega}$ maps the interval $(-1,1)$
still onto $(-1,1)$. So $f(\varphi(\Delta)\cap\partial U)=\psi((-1,1))$ and $%
f(x)=\psi(0)$ is an interior point of $f(\varphi(\Delta)\cap\partial U)$. On
the other hand, for $y\in\psi(\Delta)$, the number of its pre-image points
in $\varphi(\Delta)\cap U$ is equal to the number of pre-image points of $%
\psi^{-1}(y)$ by mapping $\zeta\mapsto\zeta^{\omega}$ in $\Delta^+$. Thus,
if $y$ lies in $\psi(\Delta^+)$, this number is $(\omega+1)/2$ and if $y$
lies on $\psi(\Delta^-)$, this number is $(\omega-1)/2$.

When $\omega$ is even, $\zeta\mapsto\zeta^{\omega}$ maps the interval $%
(-1,1) $ onto $[0,1)$. Similarly we know that $f(x)$ is an endpoint of $%
f(\varphi(\Delta)\cap\partial U)$. In this case, $(f,\partial U)$ is folded
at $f(x)$. Then for either $y\in\psi(\Delta^+)$ or $y\in\psi(\Delta^-)$, the
number of its pre-image points in $\varphi(\Delta)\cap U$ is equal to $%
\omega/2$.

Summarizing above discussion, we give the definition of order.

\begin{definition}
Let $(f,\overline{U})$ be a surface and $x\in\overline{U}$, $\omega$ be the
positive integer determined by Lemma \ref{lm:stoilowloc}. If $x\in U$, the
\emph{order} $v_f(x)$ of $(f,\overline{U})$ at $x$ equals $\omega$, and if $%
x\in\partial U$, the \emph{order} $v_f(x)$ is the minimum integer not less
than $\omega/2$.
\end{definition}

\begin{remark}
\label{rm:liftcount} The order defined here comes from the number of
pre-image points, it can be also used to count path lifts. For a surface $(f,%
\overline{U})$, $x\in\overline{U}$, and a curve $\gamma$ on $S$ starting
from $f(x)$, we consider the $f$-lift of $\gamma$ at $x$, which means a
curve $\widetilde{\gamma}$ with initial point $x$ satisfying $f\circ%
\widetilde{\gamma}=\gamma$ and $\widetilde{\gamma}^{\circ}\subset U$. The
existence of lifts and their number depend on the position of $\gamma$.

However, whether $x\in U$ or $x\in\partial U$, the number of $f$-lifts at $x$
is not greater than $v_f(x)$. In fact, when $\gamma$ is chosen properly, the
number of such lifts is exactly $v_f(x)$.
\end{remark}

\begin{definition}
Let $\Sigma=(f,\overline{U})$ be a surface, a point $x\in\overline{U}$ is
called a \emph{branch point} of $f$ (or $\Sigma$) if $v_f(x)>1$, otherwise
called a \emph{regular point}. $y\in S$ is called a \emph{branch value} if
there exists a branch point $x$ such that $f(x)=y$.
\end{definition}

\begin{corollary}
\label{cr:reghomeo} Let $\Sigma =(f,\overline{U})$ be a surface, $x\in
\overline{U}$.

(1) If $f$ restricted to some neighborhood of $x$ in $\overline{U}$ is a
homeomorphism, then $x\in \overline{U}$ is a regular point of $f$.

(2) If $x\in U$ and it is a regular point of $f$, then $f$ restricted to
some neighborhood of $x$ in $U$ is a homeomorphism.
\end{corollary}

At the end of this section, we will introduce a method of surface
reconstruction. For this purpose, let us see a simple example. Consider a
surface $\Sigma=(f,\overline{\Delta})$, and two paths $\tau^+$, $\tau^-$
contained in $\overline{\Delta}$ whose image are
\begin{align*}
\tau^+&=\{z\in\overline{\Delta^+}:\lvert z-1/2\rvert=1/2\}, & \tau^-&=\{z\in%
\overline{\Delta^-}:\lvert z-1/2\rvert=1/2\},
\end{align*}
both of which have initial point $1\in\overline{\Delta}$ and terminal point $%
0\in\overline{\Delta}$. The curve $\tau^+-\tau^-$ enclose the domain $%
D=\{z\in\mathbb{C}:\lvert z-1/2\rvert<1/2\}$.

If $f$ maps $\tau^+$, $\tau^-$ onto the same simple curve, for example, $f$
satisfies $f(z)=f(\bar{z})$ on $\partial D$, then the two ``edges'' $\tau^+$%
, $\tau^-$ can be topologically pasted. In this way, two new surfaces can be
obtained by sewing along $(f,\tau^+)\sim(f,\tau^-)$.

For $(f,\overline{D})$, we can find a homeomorphism $\varphi_0:D\to\overline{%
\mathbb{C}}\setminus[0,1]$, which can be continuously extended to $\overline{%
D}$ and maps both $\tau^+$, $\tau^-$ onto $[0,1]$ homeomorphically, such
that there exists a continuous mapping $f_0:\overline{\mathbb{C}}\to S$
which satisfies $f=f_0\circ\varphi_0$. Since $f$ is OPCOFO, so is $f_0$, and
we get a new surface $\Sigma_0=(f_0,\overline{\mathbb{C}})$. Similarly, for $%
(f,\overline{\Delta\setminus D})$, we can find a homeomorphism $%
\varphi_1:\Delta\setminus\overline{D}\to\Delta\setminus[0,1]$, which can be
continuously extended to $\overline{\Delta\setminus D}$, maps $%
\partial\Delta $ as identity and maps both $\tau^+$ and $\tau^-$ onto $[0,1]$
homeomorphically, to get another surface $\Sigma_1=(f_1,\overline{\Delta})$
where $f=f_1\circ\varphi_1$.

This also shows how we construct new surfaces by sewing in general. We can
expect to get ``simpler'' surfaces in this way. That is why we introduce
this method.

\begin{lemma}
\label{lm:sew} Let $\Sigma=(f,\overline{\Delta})$ be a surface, $%
\tau^+(x_0,x_1)$, $\tau^-(x_0,x_1)$ be two paths from $x_0$ to $x_1$ in $%
\overline{\Delta}$, such that $f\circ\tau^+$, $f\circ\tau^-$ are equivalent
paths consisting of line segments, in other word,
\begin{equation*}
(f,\tau^+)\sim(f,\tau^-),
\end{equation*}
which means there exists an increasing homeomorphism $\varphi:[0,1]\to[0,1]$
such that $f\circ\tau^+=f\circ\tau^-\circ\varphi$.

Assume $\tau ^{+}-\tau ^{-}$ is a Jordan curve enclosing a Jordan domain $%
D\subset \Delta $, such that one of the following conditions hold:

(1) $\#(\partial D)\cap (\partial \Delta )<2$.

(2) As in Figure \ref{lm:sew-fig}.\subref{lm:sew-fig1}, $\#(\partial D)\cap
(\partial \Delta )=2$, and $\tau ^{+}(t)$ and $\tau ^{-}(\varphi (t))$ can
not be contained in $\partial \Delta $ simultaneously for any $t\in (0,1)$.

Then there exists a neighborhood $V$ of $(\partial \Delta )\setminus
(\partial D)$ in $\overline{\Delta }\setminus ((\partial D)\cap (\partial
\Delta ))$, and a simple curve $\tau ^{\ast }(x_{0},x_{1})$ contained in $%
\overline{\Delta }$, as in Figure \ref{lm:sew-fig}.\subref{lm:sew-fig4},
such that two surfaces can be constructed:

(i) $\Sigma _{0}=(f_{0},\overline{\mathbb{C}})$ which satisfies $%
f=f_{0}\circ \varphi _{0}$ for a continuous mapping $\varphi _{0}:\overline{D%
}\rightarrow \overline{\mathbb{C}}$;

(ii) $\Sigma _{1}=(f_{1},\overline{\Delta })$ which satisfies $f=f_{1}\circ
\varphi _{1}$ for a continuous mapping $\varphi _{1}:\overline{\Delta
\setminus D}\rightarrow \overline{\Delta }$.

Moreover, $\varphi _{0}$ restricted to $D$ is a homeomorphism onto $%
\overline{\mathbb{C}}\setminus \tau ^{\ast }$, $\varphi _{1}$ restricted to $%
\Delta \setminus D$ is a homeomorphism onto $\Delta \setminus \tau ^{\ast }$%
. Four restrictions $\varphi _{0}:\tau ^{+}\rightarrow \tau ^{\ast }$, $%
\varphi _{0}:\tau ^{-}\rightarrow \tau ^{\ast }$, $\varphi _{1}:\tau
^{+}\rightarrow \tau ^{\ast }$, $\varphi _{1}:\tau ^{-}\rightarrow \tau
^{\ast }$ are all homeomorphism. Finally, $\varphi _{1}$ restricted to $V$
is identical.
\end{lemma}

\begin{proof}
Without losing generality, we may assume that $\varphi:[0,1]\to[0,1]$ is
identical, for otherwise we may replace $\tau^-$ by $\tau^-\circ\varphi$.

Firstly, $\tau^+(t)\mapsto\tau^-(t)$ gives the homeomorphism $\widetilde{%
\varphi}:\tau^+\to\tau^-$. What we need to do is suitably choose the
``suture line'' $\tau^{\ast}$, to give homeomorphism $\varphi_0$, $\varphi_1$
from $\tau^+$, $\tau^-$ onto $\tau^{\ast}$, and then they are defined on $%
\partial D=\tau^+-\tau^-$ and can be further extended to the whole $%
\overline{D}$ and $\overline{\Delta\setminus D}$ respectively.

(1) When $(\partial D)\cap (\partial \Delta )$ is contained in $\tau ^{+}$
or $\tau ^{-}$, without losing generality, assume $\partial D\cap \partial
\Delta \subset \tau ^{+}$, and let $\varphi _{1}(\tau ^{+}(t))=\varphi
_{1}(\tau ^{-}(t))=\tau ^{+}(t)$ for $t\in \lbrack 0,1]$. By this way $%
\varphi _{1}$ is defined on $\partial D$, and restricted to $\tau ^{+}$ and $%
\tau ^{-}$ are both homeomorphism onto $\tau ^{\ast }$, where $\tau ^{\ast
}=\tau ^{+}$ actually.

(2) Assume $(\partial D)\cap (\partial \Delta )\nsubseteq \tau ^{+}$ and $%
(\partial D)\cap (\partial \Delta )\nsubseteq \tau ^{-}$, says, $(\partial
D)\cap (\partial \Delta )$ consists of two points lied in $(\tau
^{+})^{\circ }$ and $(\tau ^{-})^{\circ }$ respectively. Then we can define
a mapping $\varphi _{1}:\partial \Delta \rightarrow $ Let $t^{+}$, $t^{-}$
be numbers in $(0,1)$ such that $\tau ^{+}(t^{+})\in \partial \Delta $, $%
\tau ^{-}(t^{-})\in \partial \Delta $. Then $t^{+}\neq t^{-}$, and without
loss of generality, assume $t^{+}<t^{-}$. Let $\widetilde{\tau }%
:[0,1]\rightarrow \overline{D}$ be a simple path from $\tau ^{+}(t^{+})$ to $%
\tau ^{-}(t^{-})$ with $\widetilde{\tau }\cap \partial D=\{\tau
^{+}(t^{+}),\tau ^{-}(t^{-})\}$ and $\widetilde{\tau }^{\circ }\subset D$.
Then define
\begin{align*}
\varphi _{1}(\tau ^{+}(t))=\varphi _{1}(\tau ^{-}(t))& =\tau ^{+}(t), & &
t\in \lbrack 0,t^{+}]; \\
\varphi _{1}(\tau ^{+}(t))=\varphi _{1}(\tau ^{-}(t))& =\widetilde{\tau }%
((t-t^{+})/(t^{-}-t^{+})), & & t\in \lbrack t^{+},t^{-}]; \\
\varphi _{1}(\tau ^{+}(t))=\varphi _{1}(\tau ^{-}(t))& =\tau ^{-}(t), & &
t\in \lbrack t^{-},1].
\end{align*}%
By this way $\varphi _{1}$ is defined on $\partial D$, and restricted to $%
\tau ^{+}$ and $\tau ^{-}$ are both homeomorphism onto $\tau ^{\ast }=\tau
^{+}(x_{0},\widetilde{\tau }(0))+\widetilde{\tau }+\tau ^{-}(\widetilde{\tau
}(1),x_{1})$, where $\tau ^{+}(x_{0},\widetilde{\tau }(0))$ means the arc of
$\tau ^{+}$ from $x_{0}$ to $\widetilde{\tau }(0)$ and $\tau ^{-}(\widetilde{%
\tau }(1),x_{1})$ means the arc of $\tau ^{-}$ from $\widetilde{\tau }(1)$
to $x_{1}$.

\begin{figure}[tbp]
\subcaptionbox{\label{lm:sew-fig1}}{\begin{tikzpicture}[scale=2.5]
\fill[lightgray] (0,0) circle (1);
\fill[white] (0.6,0.8)--(-0.3,0.1)--(-0.8,0)--(-0.6,-0.8)--(0.3,-0.1)--(0.8,0)--cycle;
\draw (0.8,0) node[below] {$x_0$} -- (0.6,0.8) node[above right] {$\tau^+(t^+)$};
\draw (0.6,0.8) -- node[above] {$\tau^+$} (-0.3,0.1) node[above] {$\tau^+(t^-)$} -- (-0.8,0) node[above] {$x_1$};
\draw (0.8,0) -- (0.3,-0.1) node[below] {$\tau^-(\varphi(t^-))$} -- node[below] {$\tau^-$} (-0.6,-0.8) node[below left] {$\tau^-(\varphi(t^-))$};
\draw (-0.6,-0.8)--(-0.8,0);
\draw[dashed] (0.6,0.8) -- node[left] {$\widetilde{\tau}$} (-0.6,-0.8);
\draw (0,0) circle (1);
\end{tikzpicture}}
\subcaptionbox{}{\begin{tikzpicture}[scale=2.5]
\fill[lightgray] (0,0) circle (1);
\fill[white] (0.6,0.8)--(-0.4,-0.2)--(-0.8,0)--(-0.6,-0.8)--(0.4,0.2)--(0.8,0)--cycle;
\draw (0.8,0) node[below] {$x_0$} -- (0.6,0.8) node[above right] {$\tau^+(t^+)$};
\draw (0.6,0.8) -- (-0.4,-0.2) -- (-0.8,0) node[above] {$x_1$};
\draw (0.8,0) -- (0.4,0.2) -- (-0.6,-0.8) node[below left] {$\tau^-(\varphi(t^-))$};
\draw (-0.6,-0.8)--(-0.8,0);
\draw[dashed] (0.6,0.8)--(-0.6,-0.8);
\draw (0,0) circle (1);
\end{tikzpicture}}
\subcaptionbox{}{\begin{tikzpicture}[scale=2.5]
\fill[lightgray] (0,0) circle (1);
\fill[white] (0.6,0.8)--(-0.5,-0.5)--(-0.8,0)--(-0.6,-0.8)--(0.5,0.5)--(0.8,0)--cycle;
\draw (0.8,0) node[below] {$x_0$} -- (0.6,0.8) node[above right] {$\tau^+(t^+)$};
\draw (0.6,0.8) -- (-0.5,-0.5) -- (-0.8,0) node[above] {$x_1$};
\draw (0.8,0) -- (0.5,0.5) -- (-0.6,-0.8) node[below left] {$\tau^-(\varphi(t^-))$};
\draw (-0.6,-0.8)--(-0.8,0);
\draw[dashed] (0.6,0.8)--(-0.6,-0.8);
\draw (0,0) circle (1);
\end{tikzpicture}}
\subcaptionbox{\label{lm:sew-fig4}}{\begin{tikzpicture}[scale=2.5]
\fill[lightgray] (0,0) circle (1);
\draw (0,0) circle (1);
\node[above right] at (0.6,0.8) {$\tau^+(t^+)$};
\node[below left] at (-0.6,-0.8) {$\tau^-(\varphi(t^-))$};
\clip (0,0) circle (1);
\draw[double] (0.8,0) node[below] {$x_0$} -- (0.6,0.8) -- (-0.6,-0.8) -- (-0.8,0) node[above] {$x_1$};
\end{tikzpicture}}
\caption{ }
\label{lm:sew-fig}
\end{figure}

After $\varphi _{1}$ is defined on $\partial D$, let $\varphi _{0}=\varphi
_{1}$ on $\partial D$, and extend $\varphi _{0}$ to a continuous mapping on $%
\overline{D}$ which restricted to $D$ is a homeomorphism onto $\overline{%
\mathbb{C}}\setminus \tau ^{\ast }$. Then $\varphi _{0}$ satisfies our
conclusion.

Further, choose the neighborhood $V$ of $(\partial \Delta )\setminus
(\partial D)$ such that $V$ has no intersection with a neighborhood of $%
\overline{D}\setminus ((\partial D)\cap (\partial \Delta ))$ in $\overline{%
\Delta }$. Then extend $\varphi _{1}$ to a continuous mapping on $\overline{%
\Delta \setminus D}$ which restricted to $\Delta \setminus D$ is a
homeomorphism onto $\Delta \setminus \tau ^{\ast }$, and restricted to $V$
is identical. Similarly $\varphi _{1}$ satisfies our conclusion.
\end{proof}

When $\tau^+(t)$ and $\tau^-(\varphi(t))$ are contained in $\partial\Delta$
simultaneously for some $t\in(0,1)$, this lemma can not be used. In this
case, the quotient space of $\overline{\Delta\setminus D}$ obtained by
identifying $\tau^+(t)$ and $\tau^-(\varphi(t))$ for $t\in(0,1)$ is
topologically not one disk anymore.

\begin{remark}
Let $\Sigma $ be a surface, and $\Sigma _{0}$, $\Sigma _{1}$ be surfaces
obtained by Lemma \ref{lm:sew}. It is obviously that
\begin{equation*}
\deg _{\min }(\Sigma )=\deg _{\min }(\Sigma _{0})+\deg _{\min }(\Sigma _{1})%
\mathrm{\ and\ }\deg _{\min }(\Sigma _{1})<\deg _{\min }(\Sigma ).
\end{equation*}
\end{remark}

\section{Reducing surfaces and proof of the main theorem}

Now we will complete the proof of our main result (Theorem \ref{thr:main}).
It is $\deg_{\max}(f)$ that this theorem refers, but actually:

\begin{lemma}
\label{lm:arg} We have $\deg_{\max}(\Sigma)-\deg_{\min}(\Sigma)\le m$ for
any $\Sigma\in\mathcal{F}_r(L,m)$.
\end{lemma}

\begin{proof}
Assume $\Sigma=(f,\overline{\Delta})\in\mathcal{F}_r(L,m)$. By Sto\"{\i}%
low's theorem, $f$ is meromorphic up to a homeomorphic transformation, and
thus the argument principle applies to $f$.

Let points $y$ and $y_0$ in $S\setminus\partial\Sigma$ satisfy $%
\#f^{-1}(y)=\deg_{\min}(f)$ and $\#f^{-1}(y_0)=\deg_{\max}(f)$. Up to a
fractional linear mapping on $S$, without loss of generality, we may assume
that $y=\infty$ and $y_0=0$ in $S$. Recall the $\mathcal{F}(L,m)$-partition
of $\Sigma$ in \eqref{def:FLM-1},
\begin{equation*}
\partial\Delta=\alpha_1(\mathfrak{p}_1,\mathfrak{p}_2)+\alpha_2(\mathfrak{p}%
_2,\mathfrak{p}_3)+\dots+\alpha_m(\mathfrak{p}_m,\mathfrak{p}_1).
\end{equation*}
By the argument principle, we have
\begin{equation*}
2\pi\deg_{\max}(f)-2\pi\deg_{\min}(f)=\theta_1+\theta_2+\dots+\theta_m,
\end{equation*}
where $\theta_j$ is the change in the argument of $f$ over $\alpha_j$. Since
$f$ maps each $\alpha_j$ to a simple circular arc, $\lvert\theta_j\rvert\le2%
\pi$. Thus we get $\deg_{\max}(f)-\deg_{\min}(f)\le m$.
\end{proof}

To prove Theorem \ref{thr:main}, our aim is to find a way to decrease the
degree of surfaces. The method of sewing introduced in Lemma \ref{lm:sew}
can help us. For a surface $\Sigma=(f,\overline{\Delta})$, we hope to take
an inverse of $f$ on $S\setminus\beta$, where $\beta$ is a simple path on $S$%
, and hope that the $f$-lifts of $\beta$ may give the two paths that satisfy
the conditions of Lemma \ref{lm:sew}.

Because we need to calculate $\overline{n}(\Sigma)=\overline{n}(\Sigma,E_q)$%
, the path $\beta$ can be chosen as follows.

\begin{lemma}
For the special set $E_q$ on $S$, we can order the $q$ points of $E_q$ as $%
a_1$, $a_2$, $\dots$, $a_q$, such that
\begin{equation*}
\beta=\beta_{E_q}=\beta(a_1,a_2)+\beta(a_2,a_3)+\dots+\beta(a_{q-1},a_q)
\end{equation*}
is a simple path, where each $\beta(a_j,a_{j+1})$ is a line segment from $%
a_j $ to $a_{j+1}$.
\end{lemma}

\begin{proof}
Let $b$ and $b^{\ast}$ be fixed antipodal points on $S\setminus E_q$, such
that any great circle passing through $b$ and $b^{\ast}$ contains at most
one point of $E_q$. Consider a line segment $l$ on $S$ from $b$ to $b^{\ast}$%
. Let $l$ fix the endpoints and rotate continuously on $S$. Assume that it
encounters $a_1$, $a_2$, $\dots$, $a_q$ in $E_q$ successively. In this way,
we can take $\beta=\beta(a_1,a_2)+\beta(a_2,a_3)+\dots+\beta(a_{q-1},a_q)$
as the path consisting of $q-1$ line segments which satisfies requirements.
\end{proof}

Next, our task is to find the inverse of $f$ on $S\setminus\beta$ for a
surface $\Sigma=(f,\overline{\Delta})$. This can be achieved locally and so
that the standard concept of triangulated surfaces is useful in our study.

\begin{definition}
A \emph{topological triangle} $F$ on $S$ is a closed Jordan domain with
three distinguished points on its boundary, which are called \emph{vertices}
of $F$. The vertices divide $\partial F$ into three edges, each of which is
called an \emph{edge} of $F$.

Let $\mathcal{T}$ be a collection of topological triangles. The vertices and
the edges of the topological triangles of $\mathcal{T}$ are called vertices
and edges of $\mathcal{T}$, and the topological triangles of $\mathcal{T}$
are called \emph{faces} of $\mathcal{T}$.

Let $U\subset S$ be a domain. A finite collection $\mathcal{T}$ of
topological triangles in $\overline{U}$ is called a \emph{triangulation} of $%
\overline{U}$ if the following hold:

(1) $\overline{U}=\cup _{F\in \mathcal{T}}F$;

(2) for any two distinct edges $e_{1}$ and $e_{2}$ of $\mathcal{T}$, $%
e_{1}\cap e_{2}$ is empty or is a singleton which is the common endpoint of
them;

(3) for any two distinct faces $F_{1}$ and $F_{2}$ of $\mathcal{T}$, $%
F_{1}\cap F_{2}$ is empty, or is a singleton which is a common vertex of $%
F_{1}$ and $F_{2}$, or is a common edge of $F_{1}$ and $F_{2}$.
\end{definition}

\begin{lemma}
\label{lm:trian} Let $k$ be a natural number. Then there exists a constant $%
C_{0}(k)$ which depends only on $k$, such that for any $k$ simple circular
arcs $c_{1}(x_{1},y_{1})$, $c_{2}(x_{2},y_{2})$, $\dots $, $%
c_{k}(x_{k},y_{k})$ on $S$, there exists a triangulation $\mathcal{T}$ of $S$
satisfying the following conditions:

(1) for each simple circular arc $c_{j}(x_{j},y_{j})$, its initial point $%
x_{j}$ and terminal point $y_{j}$ are vertices of $\mathcal{T}$;

(2) for each $j=1,2,\dots ,k$, $c_{j}$ is the union of some edges of $%
\mathcal{T}$;

(3) the number of faces of $\mathcal{T}$ is not greater than $C_{0}(k)$.
\end{lemma}

\begin{proof}
Let $E=c_1\cup c_2\cup\dots\cup c_k$. Without losing generality, we can
assume that $E$ is connected, for otherwise we can add at most $k$ simple
circular arcs to $E$. It is clear that we have a collection of edges
\begin{equation*}
\mathcal{E}=\{e_j\}_{j=1}^{k_1},
\end{equation*}
such that $E=e_1\cup e_2\cup\dots\cup e_{k_1}$, where each $e_j$ is a simple
circular arc which is a arc of some $c_{j^{\prime }}$, $1\le j^{\prime }\le
k $. Moreover, for any two distinct $e_{j_1}$ and $e_{j_2}$, $e_{j_1}\cap
e_{j_2}$ is empty or is a singleton which is the common endpoint of them.
Noticing that any two distinct circles on $S$ have at most two common
points, we can assume that $k_1\le C_1(k)$ for some constant $C_1(k)$ which
depends only on $k$.

Let $\mathcal{V}$ be the set of all endpoints of edges in $\mathcal{E}$.
Since $E$ is connected, we can assume that
\begin{equation*}
S\setminus E=U_1\cup U_2\cup\dots\cup U_{k_2},
\end{equation*}
where $U_1$, $U_2$, $\dots$, $U_{k_2}$ are all connected components of $%
S\setminus E$, and $k_2\le C_2(k)$ for some constant $C_2(k)$ which depends
only on $k$. In fact we have $C_2(k)\le2C_1(k)$, since each edge of $%
\mathcal{E}$ is shared by at most two components of $S\setminus E$. Then
each $U_j$ is a simply connected domain with boundary consisting of at most $%
2k_1$ edges in $\mathcal{E}$, and without losing generality, we can assume
that the number of these edges is at least three for each $U_j$, for
otherwise we can replace the edges in $\mathcal{E}$ by some arcs. Each $U_j$
is conformally equivalent to $\Delta$, and so $\overline{U_j}$ has a
triangulation $\mathcal{T}_j$, whose vertex set is $(\mathcal{V}\cap\partial
U_j)\cup\{v_j\}$, where $v_j\in U_j$. Then $\mathcal{T}=\cup_{j=1}^{k_2}%
\mathcal{T}_j$ is the desired triangulation with at most $C_0(k)\le
2C_1(k)C_2(k)$ faces.
\end{proof}

Now, we show the following.

\begin{lemma}
\label{lm:triansur} Let $\Sigma=(f,\overline{\Delta})\in\mathcal{F}_r(L,m)$
with $\mathcal{F}(L,m)$-partitions \eqref{def:FLM-1} and \eqref{def:FLM-2}.
Then there exists a triangulation $\mathcal{T}$ of $S$ such that:

(1) every point of $E_{q}\cup \{p_{j}\}_{j=1}^{m}$ is a vertex of $\mathcal{T%
}$, where $\{p_{j}\}_{j=1}^{m}$ is the set of points appeared in %
\eqref{def:FLM-2};

(2) for each $j=1,2,\dots ,m$, $c_{j}$ in \eqref{def:FLM-2} is a union of
some edges of $\mathcal{T}$, and so is the arc $\beta (a_{j},a_{j+1})$ of $%
\beta $ for each $j=1,2,\dots ,q-1$;

(3) all faces of $\mathcal{T}$ can be arranged as $F_{1}$, $F_{2}$, $\dots $%
, $F_{k_{0}}$, where $k_{0}\leq C_{0}(m+q)$ for the constant $C_{0}(m+q)$
appeared in Lemma \ref{lm:trian}, such that for each $j$,
\begin{equation*}
F_{j}^{\ast }=(F_{1}\cup F_{2}\cup \dots \cup F_{j})^{\circ }\setminus \beta
\end{equation*}%
is a simply connected domain, and
\begin{equation*}
F_{1}^{\ast }\subset F_{2}^{\ast }\subset \dots \subset F_{k_{0}}^{\ast
}=S\setminus \beta .
\end{equation*}
\end{lemma}

For such a triangulation $\mathcal{T}$, by the definition of triangulation
it is easy to see $(F_j)^{\circ}\cap(f(\partial\Delta)\cup\beta)=\varnothing$
for each face $F_j$.

\begin{proof}
Let $\mathcal{T}$ be the triangulation obtained by Lemma \ref{lm:trian}
applied to $c_{1}(p_{1},p_{2})$, $c_{2}(p_{2},p_{3})$, $\dots $, $%
c_{m}(p_{m},p_{1})$, $\beta (a_{1},a_{2})$, $\beta (a_{2},a_{3})$, $\dots $,
$\beta (a_{q-1},a_{q})$. It is clear that we only need to prove the
existence of the arrangement in (3).

Since $\beta$ is a simple polygonal path and $S\setminus\beta$ is a simply
connected domain, by Riemann mapping theorem, there exists a continuous
mapping $\psi:\overline{\Delta}\to S$, which is conformal from $\Delta$ onto
$S\setminus\beta$, and $\partial\Delta$ has a partition $\partial\Delta=%
\tau^+-\tau^-$, such that $\psi$ maps both $\tau^+$ and $\tau^-$
homeomorphically onto $\beta$. Then we can get the induced triangulation of $%
\overline{\Delta}$ by the pullback of $\psi$, denoted by $\psi^{\ast}(%
\mathcal{T})$, which consists of every face $F$ such that $\psi(F)$ is a
face of $\mathcal{T}$. Because $\psi$ is conformal on $\Delta$, there is a
one-to-one correspondence between the faces of $\mathcal{T}$ and $%
\psi^{\ast}(\mathcal{T})$ by $\psi$.

It is clear that the last face $F_{k_{0}}^{\prime }$ of $\psi ^{\ast }(%
\mathcal{T})$ can be chosen so that $\Delta \setminus F_{k_{0}}^{\prime }$
is also a Jordan domain, since each face $F$ of $\mathcal{T}$ is either
contained in $S\setminus \beta $, or intersects $\beta $ on an edge or at a
vertex of $\mathcal{T}$. For the same reason, the order of faces of $\psi
^{\ast }(\mathcal{T})$ can be so arranged as $F_{1}^{\prime }$, $%
F_{2}^{\prime }$, $\dots $, $F_{k_{0}}^{\prime }$, such that $\Delta $, $%
\Delta \setminus F_{k_{0}}^{\prime }$, $\Delta \setminus
(F_{k_{0}-1}^{\prime }\cup F_{k_{0}}^{\prime })$, $\dots $, $\Delta
\setminus (F_{2}^{\prime }\cup \dots \cup F_{k_{0}}^{\prime })$ are all
Jordan domains. Then for each $k\leq k_{0}$, $(\cup _{j=1}^{k}F_{j}^{\prime
})^{\circ }$ is a Jordan domain. Then $F_{j}=\psi (F_{j}^{\prime })$ and $%
F_{j}^{\ast }=\psi ((F_{1}^{\prime }\cup F_{2}^{\prime }\cup \dots \cup
F_{j}^{\prime })^{\circ })$ satisfy (3).
\end{proof}

For a surface $(f,\overline{\Delta})$ and such a triangulation $\mathcal{T}$
of $S$, we can consider the triangulation of $\overline{\Delta}$ induced by
the pullback of $f$. So we can prove:

\begin{lemma}
\label{lm:GF} Let $\Sigma=(f,\overline{\Delta})\in\mathcal{F}_r(L,m)$, $%
\mathcal{G}_{\infty}(f)$ be the set consisting of all univalent branches of $%
f^{-1}$ defined on $S\setminus\beta$, such that for each $g\in\mathcal{G}%
_{\infty}(f)$,
\begin{equation*}
\#(\overline{g(S\setminus\beta)}\cap\partial\Delta)<+\infty.
\end{equation*}
Then we have
\begin{equation*}
\#\mathcal{G}_{\infty}(f)\ge \deg_{\min}(f)-d_{\infty},
\end{equation*}
where $d_{\infty}$ is a constant depends only on $m$ and $q$.
\end{lemma}

\begin{proof}
Let $\mathcal{T}$ be the triangulation of $S$ defined in Lemma \ref%
{lm:triansur}. Each branch value of $f$ is a point in $E_q$ and thus is a
vertex of $\mathcal{T}$. We can get the induced triangulation of $\overline{%
\Delta}$ by the pullback of $f$, denoted by $f^{\ast}(\mathcal{T})$. The
vertices, edges and faces of $f^{\ast}(\mathcal{T})$ are mapped
homeomorphically onto vertices, edges and faces of $\mathcal{T}$
respectively. Then $\partial\Delta$ has a partition
\begin{equation*}
\partial\Delta=\gamma_1+\gamma_2+\dots+\gamma_{d_{\infty}},
\end{equation*}
where each $\gamma_j$ is an edge of $f^{\ast}(\mathcal{T})$. Because $%
f(\partial\Delta)$ consists of $m$ simple circular arcs, each point of $%
f(\partial\Delta)$ has at most $m$ inverses on $\partial\Delta$. Thus
\begin{equation*}
d_{\infty}\le3mC_0(m+q),
\end{equation*}
where $C_0(m+q)$ is the constant mentioned in Lemma \ref{lm:triansur}.

It is clear that $f^{-1}(F_1)$ consists of at least $\deg_{\min}(f)$ faces
of $f^{\ast}(\mathcal{T})$. There are also at least $\deg_{\min}(f)$
distinct univalent branches of $f^{-1}$ defined on $F_1=\overline{F_1^{\ast}}
$. Let $G_1$ be the set of all these univalent branches of $f^{-1}$ defined
on $\overline{F_1^{\ast}}\setminus\beta$, such that each $g\in G_1$
satisfies $\#(\overline{g(F_1^{\ast})}\cap\partial\Delta)<\infty$.

For each univalent branch $g$ of $f^{-1}$ defined on $\overline{F_1^{\ast}}%
\setminus\beta$, $g\notin G_1$ can only happen in the case that one edge of $%
\overline{g(F_1^{\ast})}$ is contained in $\partial\Delta$, and there is
only one face in $f^{\ast}(\mathcal{T})$ containing this edge. Then we have
\begin{equation*}
\#G_1\ge \deg_{\min}(f)-\#\mathcal{E}_1,
\end{equation*}
where $\mathcal{E}_1$ is the subset of $\{\gamma_j\}_{j=1}^{d_{\infty}}$,
consisting of such $\gamma_j$ with $f(\gamma_j)\subset\partial F_1$.

Generally, let $G_j$ be the set of all univalent branches of $f^{-1}$
defined on $\overline{F_j^{\ast}}\setminus\beta$, such that each $g\in G_j$
satisfies $\#(\overline{g(F_j^{\ast})}\cap\partial\Delta)<\infty$. By this
definition, for each fixed $j<k_0$ and each $g\in G_j$, $g$ can be
definitely extended to $\overline{F_{j+1}^{\ast}}\setminus\beta$ across $%
\Gamma_j$, where $\Gamma_j$ is not contained in $\beta$ and is the common
edge of $\overline{F_j^{\ast}}$ and $F_{j+1}$, which is a simple arc
consisting of one or two edges of $\mathcal{T}$. In fact $f$ has no branch
value on $S\setminus\beta\supset\Gamma_j\setminus\beta\supset\Gamma_j^{%
\circ} $ and $g(\Gamma_j^{\circ})\subset\Delta$, which implies the existence
of extension.

For this extension of $g$ defined on $\overline{F_{j+1}^{\ast}}%
\setminus\beta $, it will certainly belong to $G_{j+1}$ when there is no
edge of $F_{j+1}$ mapped into $\partial\Delta$ by $g$. So we have
\begin{equation*}
\#G_{j+1}\ge\#G_j-\#\mathcal{E}_j,
\end{equation*}
where $\mathcal{E}_j\subset\{\gamma_j\}_{j=1}^{d_{\infty}}$ consists of $%
\gamma_j$ with $f(\gamma_j)\subset(\partial
F_{j+1})\setminus\Gamma_j^{\circ} $.

Inductively, we get $\mathcal{G}_{\infty}(f)=G_{k_0}$ with
\begin{equation*}
\#G_{k_0}\ge \deg_{\min}(f)-\#\mathcal{E}_1-\#\mathcal{E}_2-\dots-\#\mathcal{%
E}_{d_{\infty}-1}.
\end{equation*}
Because $\partial F_1$, $(\partial F_2)\setminus\Gamma_1^{\circ}$, $\dots$, $%
(\partial F_{k_0})\setminus\Gamma_{k_0-1}^{\circ}$ consist of some distinct
edges of $\mathcal{T}$, we know
\begin{equation*}
\#\mathcal{E}_1+\#\mathcal{E}_2+\dots+\#\mathcal{E}_{d_{\infty}-1}\le
d_{\infty},
\end{equation*}
which proves the conclusion.
\end{proof}

We obtain the lower bound of the number of elements in $\mathcal{G}%
_{\infty}(f)$, and now need to have a basic discussion on the properties of
these elements.

\begin{lemma}
\label{lm:GFin} Let $\Sigma =(f,\overline{\Delta })\in \mathcal{F}_{r}(L,m)$%
, $\mathcal{G}_{\infty }(f)$ be the set given by Lemma \ref{lm:GF}. If there
are distinct $g_{1}$ and $g_{2}$ in $\mathcal{G}_{\infty }(f)$, then $%
g_{1}(S\setminus \beta )\cap g_{2}(S\setminus \beta )=\varnothing $.
\end{lemma}

\begin{proof}
Let $x$ be a points in $g_{1}(S\setminus \beta )\cap g_{2}(S\setminus \beta
) $, which implies $g_{1}(f(x))=g_{2}(f(x))$. Because $g_{1}\neq g_{2}$,
there exists $y\in S\setminus \beta $ such that $g_{1}(y)\neq g_{2}(y)$. Let
$\gamma $ be a path from $f(x)$ to $y$ with $\gamma ^{\circ }\subset
S\setminus \beta $. Then $g_{1}\circ \gamma $ and $g_{2}\circ \gamma $ give
two different $f$-lifts of $\gamma $, but $f$ has no branch value outside $%
E_{q}\subset \beta $, which implies a contradiction.
\end{proof}

\begin{lemma}
\label{lm:GFbound} Let $\Sigma=(f,\overline{\Delta})\in\mathcal{F}_r(L,m)$
and $g\in\mathcal{G}_{\infty}(f)$. Then the following hold.

(1) \label{lm:GFbound-1}$\partial g(S\setminus \beta )$ has the partition
\begin{equation*}
\partial g(S\setminus \beta )=(\tau _{1}^{+}+\tau _{2}^{+}+\dots +\tau
_{q-1}^{+})-(\tau _{1}^{-}+\tau _{2}^{-}+\dots +\tau _{q-1}^{-}),
\end{equation*}%
where $f$ restricted to $\tau ^{+}=\tau _{1}^{+}+\tau _{2}^{+}+\dots +\tau
_{q-1}^{+}$ and $\tau ^{-}=\tau _{1}^{-}+\tau _{2}^{-}+\dots +\tau
_{q-1}^{-} $ are both homeomorphisms onto $\beta $, and $f$ maps both $\tau
_{j}^{+}$ and $\tau _{j}^{-}$ onto $\beta (a_{j},a_{j+1})$ homeomorphically
for each $j $.

(2) \label{lm:GFbound-2}If $\tau ^{+}$ and $\tau ^{-}$ have a common point $%
x $, then there exists $j$ such that $x$ is a common point of $\tau _{j}^{+}$
and $\tau _{j}^{-}$. Moreover, if $x\in (\tau _{j}^{+})^{\circ }$ or $x\in
(\tau _{j}^{-})^{\circ }$, then $\tau _{j}^{+}=\tau _{j}^{-}$.

(3) \label{lm:GFbound-3}There exist arcs of $\tau ^{+}$ and $\tau ^{-}$,
respectively given by
\begin{equation*}
(\tau ^{\ast })^{+}=\tau _{j_{0}}^{+}+\tau _{j_{0}+1}^{+}+\dots +\tau
_{j_{0}+k}^{+}, (\tau ^{\ast })^{-}=\tau _{j_{0}}^{-}+\tau
_{j_{0}+1}^{-}+\dots +\tau _{j_{0}+k}^{-},
\end{equation*}%
such that $\tau ^{\ast }=(\tau ^{\ast })^{+}-(\tau ^{\ast })^{-}$ is a
Jordan curve. Moreover, $g(S\setminus \beta )$ is contained in the domain
enclosed by $\tau ^{\ast }$.
\end{lemma}

\begin{proof}
Because $S\setminus \beta $ is conformally equivalent to $\Delta $, it is
clear that $g$ can be continuously extended to either side of $\beta $. From
this we obtain two $f$-lifts $\tau ^{+}$ and $\tau ^{-}$ of $\beta $ which
are defined by the extension of $g$ on different sides of $\beta $.
Obviously $f$ restricted to $\tau ^{+}$ and $\tau ^{-}$ are both
homeomorphisms onto $\beta $. Let $\tau _{j}^{+}$ and $\tau _{j}^{-}$ be
arcs of $\tau ^{+}$ and $\tau ^{-}$ respectively, which are homeomorphically
mapped onto $\beta (a_{j},a_{j+1})$ by $f$. Now we get the partition of $%
\partial g(S\setminus \beta )=\tau ^{+}-\tau ^{-}$, and (1) is proved.

To prove (2), let $x\in \tau ^{+}\cap \tau ^{-}$. We have $f(x)\in \beta $.
If $f(x)\in \beta (a_{j},a_{j+1})$, we know $x\in \tau _{j}^{+}\cap \tau
_{j}^{-}$. If $x\in (\tau _{j}^{+})^{\circ }\cup (\tau _{j}^{-})^{\circ }$,
say, $f(x)\in (\beta (a_{j},a_{j+1}))^{\circ }$, then $f(x)\notin E_{q}$,
and then $x$ is a regular point of $f$, and then by the uniqueness of $f$%
-lift, we must have $\tau _{j}^{+}=\tau _{j}^{-}$, and (2) is proved.

Now we prove (3). Because $g(S\setminus \beta )\subset \Delta $, $\tau
^{+}\neq \tau ^{-}$. For any $j$ with $\tau _{j}^{+}\neq \tau _{j}^{-}$, we
take $j^{\prime }\leq j$ as the maximum integer such that $\tau _{j^{\prime
}}^{+}$ and $\tau _{j^{\prime }}^{-}$ have the same initial point, and take $%
j^{\prime \prime }\geq j$ as the minimum integer such that $\tau _{j^{\prime
\prime }}^{+}$ and $\tau _{j^{\prime \prime }}^{-}$ have the same terminal
point. We denote $(\tau _{j}^{\ast })^{+}=\tau _{j^{\prime }}^{+}+\tau
_{j^{\prime }+1}^{+}+\dots +\tau _{j^{\prime \prime }}^{+}$ and $(\tau
_{j}^{\ast })^{-}=\tau _{j^{\prime }}^{-}+\tau _{j^{\prime }+1}^{-}+\dots
+\tau _{j^{\prime \prime }}^{-}$. Then by (1) and (2), $\tau _{j}^{\ast
}=(\tau _{j}^{\ast })^{+}-(\tau _{j}^{\ast })^{-}$ must be a Jordan curve.
Assume that $U_{j}\subset \Delta $ is the domain enclosed by $\tau
_{j}^{\ast }$ (when $\tau _{j}^{+}=\tau _{j}^{-}$, we can assume $%
U_{j}=\varnothing $).

For $j_{1}$ and $j_{2}$ such that $\tau _{j_{1}}^{\ast }\neq \tau
_{j_{2}}^{\ast }$, we may assume $j_{1}<j_{2}$. If $\tau _{j_{1}}^{\ast }$
and $\tau _{j_{2}}^{\ast }$ have a common point $x$, then by (1) and (2), $x$
must be the terminal point of $(\tau _{j_{1}}^{\ast })^{+}$ and $(\tau
_{j_{1}}^{\ast })^{-}$, and meanwhile $x$ is the initial point of $(\tau
_{j_{2}}^{\ast })^{+}$ and $(\tau _{j_{2}}^{\ast })^{-}$, and thus $x$ is
the unique common point of $\tau _{j_{1}}^{\ast }$ and $\tau _{j_{2}}^{\ast
} $. In fact, $f(x)$ is the unique common point of $f((\tau _{j_{1}}^{\ast
})^{+})$ and $f((\tau _{j_{2}}^{\ast })^{+})$, where both $f((\tau
_{j_{1}}^{\ast })^{+})=f((\tau _{j_{1}}^{\ast })^{-})$ and $f((\tau
_{j_{2}}^{\ast })^{+})=f((\tau _{j_{2}}^{\ast })^{-})$ are arcs of $\beta $.
Then for any $U_{j_{1}}$ and $U_{j_{2}}$, we have $U_{j_{1}}\subset
U_{j_{2}} $, or $U_{j_{1}}\supset U_{j_{2}}$, or $U_{j_{1}}\cap
U_{j_{2}}=\varnothing $.

Now let $U_{j_{0}}$ be \textquotedblleft maximal\textquotedblright , which
means that there is no $U_{j}$ such that $U_{j}\neq U_{j_{0}}$ and $%
U_{j}\supset U_{j_{0}}$. Since $g(S\setminus \beta )$ is connected, we get $%
g(S\setminus \beta )\subset U_{j_{0}}$. Thus $\tau ^{\ast }=\tau
_{j_{0}}^{\ast }$, $(\tau ^{\ast })^{+}=(\tau _{j_{0}}^{\ast })^{+}$ and $%
(\tau ^{\ast })^{-}=(\tau _{j_{0}}^{\ast })^{-}$ satisfy (3).
\end{proof}

The two curves $(\tau^{\ast})^+$ and $(\tau^{\ast})^-$ we find are as that
described in Lemma \ref{lm:sew}. If some other conditions can be also
satisfied, we can use Lemma \ref{lm:sew} to construct a new surface by
sewing, whose degree is smaller.

\begin{lemma}
\label{lm:sewdeg} Let $\Sigma=(f,\overline{\Delta})\in\mathcal{F}_r(L,m)$ be
a surface with $\mathcal{F}(L,m)$-partitions \eqref{def:FLM-1} and %
\eqref{def:FLM-2}, $\tau^+$ and $\tau^-$ be two curves satisfying all
conditions in Lemma \ref{lm:sew}, and moreover, $(\tau^+-\tau^-)\cap\partial%
\Delta\subset f^{-1}(E_q)\cap\{\mathfrak{p}_j\}_{j=1}^m$. Then there exists
a surface $\Sigma_1=(f_1,\overline{\Delta})\in\mathcal{F}_r(L,m)$, such that
\begin{align*}
\partial\Sigma_1&=\partial\Sigma, \\
H(\Sigma_1)&=H(\Sigma),
\end{align*}
and moreover,
\begin{equation*}
\deg_{\min}(f_1)<\deg_{\min}(f).
\end{equation*}
\end{lemma}

\begin{proof}
Let $D\subset\Delta$ be the Jordan domain with $\partial D=\tau^+-\tau^-$.
Since $(f,\tau^+)\sim(f,\tau^-)$, which means that the $f$-lifts near the
endpoints of $\tau^+$ and $\tau^-$ are not unique, these two endpoints are
branch points of $f$ (see Remark \ref{rm:liftcount}, where $f$-lifts are
also defined).

Let $\Sigma_0=(f_0,\overline{\mathbb{C}})$ and $\Sigma_1=(f_1,\overline{%
\Delta})$ be the surfaces, and $\varphi_0:\overline{D}\to\overline{\mathbb{C}%
}$ and $\varphi_1:\overline{\Delta\setminus D}\to\overline{\Delta}$ be the
continuous mappings obtained by Lemma \ref{lm:sew}. It is obviously that $%
\deg_{\min}(f_1)<\deg_{\min}(f)$ and $\partial\Sigma_1=\partial\Sigma$. We
prove $\Sigma_1\in\mathcal{F}_r(L,m)$ at first. By Lemma \ref{lm:sew}, there
exists a neighborhood $V$ of $(\partial\Delta)\setminus(\partial D)$ in $%
\overline{\Delta}\setminus((\partial D)\cap(\partial\Delta))$ such that $%
\varphi_1$ restricted to $V$ is identical. We get $\Sigma_1\in\mathcal{F}%
(L,m)$, and $\partial\Sigma_1=(f_1,\partial\Delta)=(f,\Delta)=\partial\Sigma$%
, since $(\partial D)\cap(\partial\Delta)\subset\{\mathfrak{p}_j\}_{j=1}^m$.

Assume $y\in\overline{\Delta}$ and $f_1(y)\notin E_q$. For each $%
x\in\varphi_1^{-1}(y)$, we have $f(x)=f_1(y)\notin E_q$, and $x$ is a
regular point of $f$. By Corollary \ref{cr:reghomeo}, if $x\in\Delta$, $f$
must be locally homeomorphic near $x$. So when $y\in\Delta$, $%
\varphi_1^{-1}(y)\subset\Delta$ and $y$ is also a regular point of $f_1$. On
the other hand, $(\partial D)\cap(\partial\Delta)\subset f^{-1}(E_q)$ and $%
\varphi_1$ restricted to $V$ is identical. $f_1$ has no branch points on $%
\partial\Delta$ outside $f^{-1}(E_q)$. Thus $\Sigma_1\in\mathcal{F}_r(L,m)$.

\begin{figure}[tbp]
\subcaptionbox{}{\begin{tikzpicture}[scale=2.5]
\fill[lightgray] (0,0) circle (1);
\draw[dashed] (0,0) circle (0.2);
\draw[dashed] (0,-0.6) circle (0.2);
\fill[white] (0.6,0)--(-0.6,0) arc[start angle=-180, end angle=0, radius=0.6];
\draw (0.6,0) node[right] {$x_0$} -- (-0.6,0) node[left] {$x_1$} arc[start angle=-180, end angle=0, radius=0.6];
\draw (0,0) circle (1);
\end{tikzpicture}}\qquad
\subcaptionbox{}{\begin{tikzpicture}[scale=2.5]
\fill[lightgray] (0,0) circle (1);
\draw[dashed] (0,0) circle (0.2);
\draw[double] (0.6,0) node[right] {$x_0$} -- (-0.6,0) node[left] {$x_1$};
\draw (0,0) circle (1);
\end{tikzpicture}}
\caption{Lemma \protect\ref{lm:sewdeg}}
\end{figure}

Next we just need to prove $H(\Sigma_1)=H(\Sigma)$, which is equivalent to
\begin{equation*}
(q-2)A(\Sigma_1)-4\pi\overline{n}(\Sigma_1)=(q-2)A(\Sigma)-4\pi\overline{n}%
(\Sigma).
\end{equation*}

By Sto\"ilow's theorem (Theorem \ref{thr:stoilow}), the surface $%
\Sigma_0=(f_0,\overline{\mathbb{C}})$ is equivalent to a surface given by a
meromorphic function. That is, $f_0$ is a rational function via a
homeomorphic transformation of coordinates, and Riemann-Hurwitz formula
applies to $f_0$. We denote $d=\deg_{\max}(f_0)=\deg_{\min}(f_0)$. The total
order of ramification of $f_0$ is equal to
\begin{equation*}
\sum_{x\in\overline{\mathbb{C}}}(v_{f_0}(x)-1)=2d-2.
\end{equation*}
Similarly to $f_1$, $f_0$ has also no branch value outside $E_q$. For each $%
a_j\in E_q$, consider the set $f_0^{-1}(a_j)$. We have $\#f_0^{-1}(a_j)=%
\overline{n}(\Sigma_0,a_j)$, and for each $j=1,2,\dots,q$,
\begin{equation*}
\sum_{x\in f_0^{-1}(a_j)}(v_{f_0}(x)-1)=\sum_{x\in f_0^{-1}(a_j)}v_{f_0}(x)-%
\overline{n}(\Sigma_0,a_j)=d-\overline{n}(\Sigma_0,a_j).
\end{equation*}
Summing this we get
\begin{equation*}
2d-2=\sum_{j=1}^q\sum_{x\in f_0^{-1}(a_j)}(v_{f_0}(x)-1)=\sum_{j=1}^q(d-%
\overline{n}(\Sigma_0,a_j))=qd-\overline{n}(\Sigma_0).
\end{equation*}
Thus we finally get $\overline{n}(\Sigma_0)=(q-2)d+2$.

It is clearly that $A(\Sigma)=A(\Sigma_0)+A(\Sigma_1)$, and $%
A(\Sigma_0)=4\pi d$ actually. Next, we will compute $\overline{n}(\Sigma)$
and $\overline{n}(\Sigma_0)+\overline{n}(\Sigma_1)$. For this purpose,
consider the contribution of each $x\in\overline{\Delta}$ with $f(x)\in E_q$.

\begin{enumerate}
\item If $x\in (\partial \Delta )\setminus (\tau ^{+}-\tau ^{-})$, it has no
contribution to $\overline{n}(\Sigma )$, $\overline{n}(\Sigma _{0})$ and $%
\overline{n}(\Sigma _{1})$.

\item If $x\in D$, then its contribution to $\overline{n}(\Sigma )$ is $1$, $%
\varphi _{0}(x)$ gives contribution $1$ to $\overline{n}(\Sigma _{0})$, but $%
\varphi _{1}(x)$ is undefined and has no contribution to $\overline{n}%
(\Sigma _{1})$.

\item If $x\in \Delta \setminus \overline{D}$, then its contribution to $%
\overline{n}(\Sigma )$ is $1$, $\varphi _{1}(x)$ gives contribution $1$ to $%
\overline{n}(\Sigma _{1})$, but $\varphi _{0}(x)$ is undefined and has no
contribution to $\overline{n}(\Sigma _{0})$.

\item If $x\in (\tau ^{+})^{\circ }$, then there exists exactly one point $%
x^{\ast }\in (\tau ^{-})^{\circ }$ as its corresponding point, say, $\varphi
_{0}(x)=\varphi _{0}(x^{\ast })$, $\varphi _{1}(x)=\varphi _{1}(x^{\ast })$,
and vice versa. For such a pair of points $x$, $x^{\ast }$, if none of them
lies on $\partial \Delta $, their contribution to $\overline{n}(\Sigma )$ is
$2$, and $\varphi _{0}(x)=\varphi _{0}(x^{\ast })$, $\varphi _{1}(x)=\varphi
_{1}(x^{\ast })$ give contribution $1$ to $\overline{n}(\Sigma _{0})$ and $%
\overline{n}(\Sigma _{1})$ respectively; if only one point of $x$ and $%
x^{\ast }$ lies on $\partial \Delta $, their contribution to $\overline{n}%
(\Sigma )$ is $1$, $\varphi _{0}(x)=\varphi _{0}(x^{\ast })$ gives
contribution $1$ to $\overline{n}(\Sigma _{0})$, and $\varphi _{1}(x)\in
\partial \Delta $ has no contribution to $\overline{n}(\Sigma _{1})$.

\item If $x$ is an endpoint of $\tau ^{+}$ and $x\notin \partial \Delta $,
its contribution to $\overline{n}(\Sigma )$ is $1$, and $\varphi _{0}(x)$, $%
\varphi _{1}(x)$ give contribution $1$ to $\overline{n}(\Sigma _{0})$ and $%
\overline{n}(\Sigma _{1})$ respectively. if If $x$ is an endpoint of $\tau
^{+}$ and $x\in \partial \Delta $, it has no contribution to $\overline{n}%
(\Sigma )$ and $\overline{n}(\Sigma _{1})$, and $\varphi _{1}(x)$ gives
contribution $1$ to $\overline{n}(\Sigma _{0})$.
\end{enumerate}

Thus, except in the case that $x$ is the endpoint of $\tau ^{+}$, $x$ (or
its corresponding pair) always gives same contribution to $\overline{n}%
(\Sigma )$ and $\overline{n}(\Sigma _{0})+\overline{n}(\Sigma _{1})$.
Therefore
\begin{equation*}
\overline{n}(\Sigma )=\overline{n}(\Sigma _{0})+\overline{n}(\Sigma _{1})-2.
\end{equation*}%
Because $A(\Sigma _{0})=4\pi d$ and $\overline{n}(\Sigma _{0})=(q-2)d+2$, we
finally get $H(\Sigma _{1})=H(\Sigma )$.
\end{proof}

Apparently, if we can always find such $\tau^+$ and $\tau^-$ satisfying
related conditions for $\Sigma\in\mathcal{F}_r(L,m)$ with sufficiently large
$\deg_{\min}(\Sigma)$, then Lemma \ref{lm:sewdeg} implies Theorem \ref%
{thr:main} inductively. At present, Lemma \ref{lm:GF} tells us the existence
of ``pullback'' $g$, and Lemma \ref{lm:GFbound} tells us that we can take a
Jordan domain based on $\partial g(S\setminus\beta)$. Next we just need to
find $g\in\mathcal{G}_{\infty}(f)$ such that the corresponding boundary
satisfies other conditions in Lemma \ref{lm:sewdeg}.

Our first consideration is the intersections of $\overline{g(S\setminus\beta)%
}$ and $\partial\Delta$.

\begin{definition}
Let $\Sigma=(f,\overline{\Delta})\in\mathcal{F}_r(L,m)$, $\mathcal{G}%
_{\infty}(f)$ be the set introduced in Lemma \ref{lm:GF}. For a natural
number $d$, denote $\mathcal{G}_d(f)$ the subset of $\mathcal{G}_{\infty}(f)$%
, which consists of all $g\in\mathcal{G}_{\infty}(f)$ such that
\begin{equation*}
\#(\overline{g(S\setminus\beta)}\cap\partial\Delta)\le d.
\end{equation*}
\end{definition}

\begin{lemma}
\label{lm:GFvertex} Let $\Sigma=(f,\overline{\Delta})\in\mathcal{F}_r(L,m)$
with $\mathcal{F}(L,m)$-partitions \eqref{def:FLM-1} and \eqref{def:FLM-2}.
We have $\mathcal{G}_{\infty}(f)=\mathcal{G}_{mq+m}(f)$. Moreover, for each $%
g\in\mathcal{G}_{\infty}(f)$, we have
\begin{equation*}
\overline{g(S\setminus\beta)}\cap\partial\Delta\subset f^{-1}(E_q)\cup\{%
\mathfrak{p}_j\}_{j=1}^m,
\end{equation*}
where $\mathfrak{p}_1$, $\mathfrak{p}_2$, $\dots$, $\mathfrak{p}_m$ are
appeared in \eqref{def:FLM-2}.
\end{lemma}

\begin{proof}
Let $g\in\mathcal{G}_{\infty}(f)$, and $x\in\overline{g(S\setminus\beta)}$
be a point in $\partial\Delta\setminus\{\mathfrak{p}_j\}_{j=1}^m$. Then $%
x\in\partial g(S\setminus\beta)$. Because $\overline{g(S\setminus\beta)}%
\cap\partial\Delta<\infty$, $x$ is an isolated point of $\partial
g(S\setminus\beta)\cap\partial\Delta$. Assume $x\in\alpha_j^{\circ}$ for
some $j=1,2,\dots,q$. By Definition \ref{def:FLM}, $c_j=(f,\alpha_j)$ is a
convex simple circular arc, and $f$ restricted to a neighborhood of $%
\alpha_j^{\circ}$ in $\overline{\Delta}$ is homeomorphic.

Assume $f(x)\notin E_{q}$. Then we can assume $f(x)\in l_{j}^{\circ }$ for
some $j=1,2,\dots ,q$, where $l_{j}=\beta (a_{j},a_{j+1})$. Thus there
exists a neighborhood $U$ of $x$ in $\overline{\Delta }$, where $U^{\circ }$
is a Jordan domain, such that the following hold.

(1) $f|_{\overline{U}}:\overline{U}\rightarrow f(\overline{U})$ is a
homeomorphism;

(2) \label{lm:GFvertex-ct2}$x$ is the unique point of $U\cap \overline{%
g(S\setminus \beta )}\cap \partial \Delta $;

(3) \label{lm:GFvertex-ct3}$f(\overline{U})$ is a convex closed domain;

(4) $c_{j}$ has an arc $c_{j}^{\prime }\subset \partial f(\overline{U})$,
such that $f(x)\in c_{j}^{\circ }$ and $(f|_{\overline{U}})^{-1}$ maps $%
c_{j}^{\prime }$ into $\alpha _{j}$ homeomorphically;

(5) $l_{j}$ has an arc $l_{j}^{\prime }\subset f(\overline{U})$, such that $%
f(x)\in l_{j}^{\circ }$ and $(f|_{\overline{U}})^{-1}$ maps $l_{j}^{\prime }$
into $\partial g(S\setminus \beta )$ homeomorphically.

But $c_{j}^{\prime }$ is a simple circular arc and $l_{j}^{\prime }$ is a
line segment, which implies a contradiction with (2) and (3)\footnote{%
If $D$ is a convex Jordan domain on $S$ and $l^{\circ }$ is an open line
segment in $\overline{D}$, and $l^{\circ }\cap \partial D\neq \varnothing $,
then each component of $l^{\circ }\cap \partial D$ is not a point.}. Thus $%
f(x)\in E_{q}$. We get $\overline{g(S\setminus \beta )}\cap \partial \Delta
\subset f^{-1}(E_{q})\cup \{\mathfrak{p}_{j}\}_{j=1}^{m}$.

On the other hand, $f(\partial \Delta )$ consists of $m$ circular arcs, each
point of $f(\partial \Delta )$ has at most $m$ inverses in $\partial \Delta $%
. So $\#(f^{-1}(E_{q})\cap \partial \Delta )\leq mq$ and $g\in \mathcal{G}%
_{mq+m}(f)$.
\end{proof}

\begin{lemma}
\label{lm:GF2} Let $\Sigma=(f,\overline{\Delta})\in\mathcal{F}_r(L,m)$. We
have
\begin{equation*}
\#(\mathcal{G}_{\infty}(f)\setminus\mathcal{G}_2(f))\le d_2,
\end{equation*}
where $d_2=mq+m-2$.
\end{lemma}

\begin{proof}
For each $g\in \mathcal{G}_{\infty }(f)\setminus \mathcal{G}_{2}(f)$, let $%
\mathcal{V}_{g}=\overline{g(S\setminus \beta )}\cap \partial \Delta \subset
\mathcal{V}$, and $P_{g}$ be the polygonal domain in the Euclidean space,
which is composed of points in $\mathcal{V}_{g}$ (noticing $\#\mathcal{V}%
_{g}\geq 3$).

By Lemma \ref{lm:GFin}, for distinct $g_1,g_2\in\mathcal{G}%
_{\infty}(f)\setminus\mathcal{G}_2(f)$, there is no common point of $%
g_1(S\setminus\beta)$ and $g_2(S\setminus\beta)$. $g_2(S\setminus\beta)$
must be contained in a connected component of $\Delta\setminus\overline{%
g_1(S\setminus\beta)}$, and $\mathcal{V}_{g_2}$ is contained in the closure
of a connected component of $(\partial\Delta)\setminus\mathcal{V}_{g_1}$.
Then $P_{g_1}\cap P_{g_2}=\varnothing$.

Let $\mathcal{V}=(f^{-1}(E_q)\cap\partial\Delta)\cup\{\mathfrak{p}%
_j\}_{j=1}^m$. By Lemma \ref{lm:GFvertex}, for all $g\in\mathcal{G}%
_{\infty}(f)\setminus\mathcal{G}_2(f)$, these $P_g$ are pairwise disjoint
polygonal domains with vertices in $\mathcal{V}$. We have
\begin{equation*}
\#\mathcal{V}\ge\#(\mathcal{G}_{\infty}(f)\setminus\mathcal{G}_2(f))+2.
\end{equation*}
By Lemma \ref{lm:GFvertex}, $\#\mathcal{V}\le mq+m$ and we complete the
proof.
\end{proof}

According to our discussion on branch points and path lifts, the lift near a
regular point must be unique, and we get:

\begin{lemma}
\label{lm:GF2nreg} Let $\Sigma=(f,\overline{\Delta})\in\mathcal{F}_r(L,m)$
with $\mathcal{F}(L,m)$-partitions \eqref{def:FLM-1} and \eqref{def:FLM-2}, $%
\mathcal{G}_2^{\prime }(f)$ be the subset of $\mathcal{G}_2(f)$ such that
for each $g\in\mathcal{G}_2^{\prime }(f)$,
\begin{equation*}
\partial\overline{g(S\setminus\beta)}\cap\partial\Delta\subset
f^{-1}(E_q)\cap\{\mathfrak{p}_j\}_{j=1}^m.
\end{equation*}
Then
\begin{equation*}
\#(\mathcal{G}_2(f)\setminus\mathcal{G}_2^{\prime }(f))\le d_2^{\prime },
\end{equation*}
where $d_2^{\prime }=mq+m$.
\end{lemma}

\begin{proof}
Similar to Lemma \ref{lm:GF2}, let $\mathcal{V}=(f^{-1}(E_q)\cap\partial%
\Delta)\cup\{\mathfrak{p}_j\}_{j=1}^m$. If in $\mathcal{G}_2(f)$ there exist
distinct $g_1$ and $g_2$, such that
\begin{equation*}
\overline{g_1(S\setminus\beta)}\cap\overline{g_2(S\setminus\beta)}%
\cap\partial\Delta
\end{equation*}
owns a point $x\in\mathcal{V}$, then $\beta$ gives two different $f$-lifts
with the same initial point $x$, which makes $x$ a branch point and $f(x)\in
E_q$. By Corollary \ref{cr:reghomeo} and the definition of $\mathcal{F}(L,m)$%
, we know that $x\in\{\mathfrak{p}_j\}_{j=1}^m$.

For any point $x\in\mathcal{V}$, there is at most one $g\in\mathcal{G}%
_2(f)\setminus\mathcal{G}_2^{\prime }(f)$ such that $x\in\overline{%
g(S\setminus\beta)}$. Then $\#(\mathcal{G}_2(f)\setminus\mathcal{G}%
_2^{\prime }(f))\le\#\mathcal{V}$. We have known that $\#\mathcal{V}\le mq+m$
in Lemma \ref{lm:GFvertex}, which implies the result.
\end{proof}

So far, for a surface $\Sigma=(f,\overline{\Delta})\in\mathcal{F}_r(L,m)$
and $g\in\mathcal{G}_2^{\prime }(f)$, let $(\tau^{\ast})^+$ and $%
(\tau^{\ast})^-$ be two arcs of $\partial g(S\setminus\beta)$ described in
Lemma \ref{lm:GFbound}. By comparing with the conditions in Lemma \ref%
{lm:sew} and Lemma \ref{lm:sewdeg}, the only case in which we cannot use
Lemma \ref{lm:sew} is that $((\tau^{\ast})^+\cup(\tau^{\ast})^-)\cap\partial%
\Delta$ contains two points, and $f$ maps them to the same point.

\begin{definition}
For a surface $\Sigma=(f,\overline{\Delta})\in\mathcal{F}_r(L,m)$, let $%
\mathcal{G}_2^{\prime \prime }(f)$ be the subset of $\mathcal{G}_2^{\prime
}(f)$, which consists of all $g\in\mathcal{G}_2^{\prime }(f)$ such that
\begin{equation*}
\#(\overline{g(S\setminus\beta)}\cap\partial\Delta)\ne\#f(\overline{%
g(S\setminus\beta)}\cap\partial\Delta),
\end{equation*}
say, $\overline{g(S\setminus\beta)}\cap\partial\Delta$ has two distinct
points $x_1$ and $x_2$, and $f(x_1)=f(x_2)$.
\end{definition}

According to our previous discussion, when there exists a $g\in\mathcal{G}%
_2^{\prime }(f)\setminus\mathcal{G}_2^{\prime \prime }(f)$ for $\Sigma=(f,%
\overline{\Delta})\in\mathcal{F}_r(L,m)$, by Lemma \ref{lm:sewdeg} we can
find a surface $\Sigma_1=(f_1,\overline{\Delta})\in\mathcal{F}_r(L,m)$, such
that $\partial\Sigma_1=\partial\Sigma$, $H(\Sigma_1)=H(\Sigma)$, and $%
\deg_{\min}(f_1)<\deg_{\min}(f)$. However, this will fail when $g\in\mathcal{%
G}_2^{\prime \prime }(f)$. We need to find a workaround.

\begin{lemma}
\label{lm:GF2coin} Let $\Sigma=(f,\overline{\Delta})\in\mathcal{F}_r(L,m)$
with $\mathcal{F}(L,m)$-partitions \eqref{def:FLM-1} and \eqref{def:FLM-2},
and $d_2^{\prime \prime }=m(m-1)/2$. If
\begin{equation*}
\#\mathcal{G}_2^{\prime \prime }(f)>d_2^{\prime \prime },
\end{equation*}
then there exist two paths $\tau^+(\mathfrak{p}_{j_1},\mathfrak{p}_{j_2})$, $%
\tau^-(\mathfrak{p}_{j_1},\mathfrak{p}_{j_2})$, where $\mathfrak{p}_{j_1}$, $%
\mathfrak{p}_{j_2}$ are in the $\mathcal{F}(L,m)$-partition \eqref{def:FLM-1}
and $f(\mathfrak{p}_{j_1})=f(\mathfrak{p}_{j_1})\in E_q$, such that

(1) $\tau ^{+}-\tau ^{-}$ is a Jordan curve, and $(\tau ^{+}-\tau ^{-})\cap
\partial \Delta =\{\mathfrak{p}_{j_{1}},\mathfrak{p}_{j_{2}}\}$;

(2) $f(\tau ^{+})=f(\tau ^{-})=\beta ^{\ast }$, which is a Jordan path on $S$%
, and $f$ restricted to $(\tau ^{+})^{\circ }$ and $(\tau ^{-})^{\circ }$
are both homeomorphisms onto $(\beta ^{\ast })^{\circ }$.
\end{lemma}

\begin{proof}
For distinct $j_1$ and $j_2$, consider the pair of points $\{\mathfrak{p}%
_{j_1},\mathfrak{p}_{j_2}\}$, and let $G_{j_1,j_2}$ be the subset of $%
\mathcal{G}_2^{\prime \prime }(f)$, which consists of all $g\in\mathcal{G}%
_2^{\prime \prime }(f)$ such that
\begin{equation*}
\overline{g(S\setminus\beta)}\cap\partial\Delta=\{\mathfrak{p}_{j_1},%
\mathfrak{p}_{j_2}\}.
\end{equation*}
For $g\in G_{j_1,j_2}$, we have $f(\mathfrak{p}_{j_1})=f(\mathfrak{p}_{j_2})$%
. Then $\mathcal{G}_2^{\prime \prime }(f)=\cup_{j_1,j_2}G_{j_1,j_2}$.
Because $\#\mathcal{G}_2^{\prime \prime }(f)>d_2^{\prime \prime }=m(m-1)/2$,
according to the pigeonhole principle, there is at least one set $%
G_{j_1,j_2} $ containing two distinct elements $g^+$ and $g^-$.

For this set $G_{j_1,j_2}$, assume that $f(\mathfrak{p}_{j_1})=f(\mathfrak{p}%
_{j_2})=a_{j_0}\in E_q$. Because $G_{j_1,j_2}\subset\mathcal{G}_2^{\prime
\prime }(f)$, we know that $a_{j_0}\in\beta^{\circ}$ and thus $1<j_0<q$. Now
choose a polygonal Jordan curve $\beta^{\ast}$ on $S$ with an endpoint $%
a_{j_0}$, such that $a_{j_0}$ is the only intersection point of $%
\beta^{\ast} $ and $\beta$, and $\beta^{\ast}$ separates $\beta$.

Now $g^+$ gives an $f$-lift $\tau^+$ of $\beta^{\ast}$, such that $%
(\tau^+)^{\circ}=g^+((\beta^{\ast})^{\circ})$, and its endpoints are $%
\mathfrak{p}_{j_1}$ and $\mathfrak{p}_{j_2}$. By adjusting the direction, we
can assume that $\tau^+$ has the initial point $\mathfrak{p}_{j_1}$ and the
terminal point $\mathfrak{p}_{j_2}$. Obviously $f(\tau^+)=\beta^{\ast}$ and $%
f$ restricted to $(\tau^+)^{\circ}$ is a homeomorphism onto $%
(\beta^{\ast})^{\circ}$, and $(\tau^+)^{\circ}\subset g^+(S\setminus\beta)$.

Similarly, $g^-$ gives another $f$-lift $\tau^-(\mathfrak{p}_{j_1},\mathfrak{%
p}_{j_2})$ of $\beta^{\ast}$. Because $g^+(S\setminus\beta)$ and $%
g^-(S\setminus\beta)$ have no common point, $\tau^+-\tau^-$ is a Jordan
curve.
\end{proof}

Thus, for this case we can also use Lemma \ref{lm:sewdeg}. Summarize our
work in this section and we will prove that:

\begin{lemma}
\label{lm:main} Let $\Sigma=(f,\overline{\Delta})\in\mathcal{F}_r(L,m)$ be a
surface. If $\deg_{\min}(f)>d^{\ast}$, then there exists a surface $%
\Sigma_1=(f_1,\overline{\Delta})\in\mathcal{F}_r(L,m)$, such that
\begin{align*}
\partial\Sigma_1&=\partial\Sigma, \\
H(\Sigma_1)&=H(\Sigma),
\end{align*}
and moreover,
\begin{equation*}
\deg_{\min}(f_1)<\deg_{\min}(f),
\end{equation*}
where $d^{\ast}$ is a constant depending only on $m$ and $q$.
\end{lemma}

\begin{proof}
We just need to put $d^{\ast }=d_{\infty }+d_{2}+d_{2}^{\prime
}+d_{2}^{\prime \prime }$, where $d_{\infty }$, $d_{2}$, $d_{2}^{\prime }$, $%
d_{2}^{\prime \prime }$ are constants mentioned in Lemma \ref{lm:GF}, Lemma %
\ref{lm:GF2}, Lemma \ref{lm:GF2nreg} and Lemma \ref{lm:GF2coin}. Assume $%
\deg _{\min }(f)>d^{\ast }$. Then by Lemma \ref{lm:GF}, we have
\begin{equation*}
\#\mathcal{G}_{\infty }(f)\geq d^{\ast }-d_{\infty }=d_{2}+d_{2}^{\prime
}+d_{2}^{\prime \prime }.
\end{equation*}%
By Lemma \ref{lm:GF2}, we have
\begin{equation*}
\#(\mathcal{G}_{\infty }(f)\setminus \mathcal{G}_{2}(f))\leq d_{2}\mathrm{\
and\ }\#\mathcal{G}_{2}(f)>d_{2}^{\prime }+d_{2}^{\prime \prime }.
\end{equation*}%
By Lemma \ref{lm:GF2nreg}, we have
\begin{equation*}
\#(\mathcal{G}_{2}(f)\setminus \mathcal{G}_{2}^{\prime }(f))\leq
d_{2}^{\prime }\mathrm{\ and\ }\#\mathcal{G}_{2}^{\prime }(f)>d_{2}^{\prime
\prime }.
\end{equation*}

Assume $\mathcal{G}_2^{\prime }(f)\ne\mathcal{G}_2^{\prime \prime }(f)$.
Then there exists a $g\in\mathcal{G}_2^{\prime }(f)\setminus\mathcal{G}%
_2^{\prime \prime }(f)$. Let $(\tau^{\ast})^+$, $(\tau^{\ast})^-$ be two
curves obtained in Lemma \ref{lm:GFbound}. They satisfy the conditions in
Lemma \ref{lm:sewdeg}, and the surface $\Sigma_1$ can be constructed by
Lemma \ref{lm:sewdeg}. Conversely, if $\mathcal{G}_2^{\prime }(f)=\mathcal{G}%
_2^{\prime \prime }(f)$, then $\#\mathcal{G}_2^{\prime \prime
}(f)>d_2^{\prime \prime }$. And by Lemma \ref{lm:GF2coin}, we also have two
curves $\tau^+$, $\tau^-$, which satisfy the conditions in Lemma \ref%
{lm:sewdeg}. We can construct $\Sigma_1$ as well. So in either case, the
required surface $\Sigma_1$ must exist.
\end{proof}

Finally, with Lemma \ref{lm:arg}, it is clear that Theorem \ref{thr:main} is
a corollary of Lemma \ref{lm:main} inductively.

\end{document}